\documentclass{amsart}
\usepackage{amsmath}
\usepackage{amsfonts}
\usepackage{amssymb}
\usepackage{color}
\usepackage{graphics,graphicx}
\usepackage{hyperref}
\usepackage{amssymb,latexsym,amsthm,xypic}

\newtheorem{thm}{Theorem}[section]
\newtheorem{cor}[thm]{Corollary}
\newtheorem{lem}[thm]{Lemma}
\newtheorem{prop}[thm]{Proposition}

\theoremstyle{definition}
\newtheorem{defn}{Definition}[section]

\theoremstyle{remark}
\newtheorem{rem}{Remark}[section]

\newcommand{\be}{\begin{equation}}
\newcommand{\ee}{\end{equation}}
\newcommand{\bea}{\begin{eqnarray}}
\newcommand{\eea}{\end{eqnarray}}
\newcommand{\ben}{\begin{eqnarray*}}
\newcommand{\een}{\end{eqnarray*}}
\newcommand{\bt}{\begin{split}}
\newcommand{\et}{\end{split}}
\newcommand{\bet}{\begin{equation}}
\newcommand{\mc}{\mathbb{C}}
\newcommand{\mr}{\mathbb{R}}
\newcommand{\ra}{\rightarrow}

\begin{document}

\title[New characterizations of plurisubharmonic functions]{New characterizations of plurisubharmonic functions and positivity of direct image sheaves}
\date{}
\author[F. Deng]{Fusheng Deng}
\address{Fusheng Deng: \ School of Mathematical Sciences, University of Chinese Academy of Sciences\\ Beijing 100049, P. R. China}
\email{fshdeng@ucas.ac.cn}
\author[Z. Wang]{Zhiwei Wang}
\address{ Zhiwe Wang: \ School
of Mathematical Sciences\\Beijing Normal University\\Beijing\\ 100875\\ P. R. China}
\email{zhiwei@bnu.edu.cn}
\author[L. Zhang]{Liyou Zhang}
\address{ Liyou Zhang: \ School of Mathematical Sciences\\Capital Normal University\\Beijing \\100048\\
P. R.  China}
\email{zhangly@cnu.edu.cn}
\author[X. Zhou]{Xiangyu Zhou}
\address{Xiangyu Zhou:\ Institute of Mathematics, AMSS, and Hua Loo-Keng Key Laboratory of Mathematics, Chinese Academy of Sciences, Beijing 100190,
China; School of Mathematical Sciences, University of Chinese
Academy of Sciences, Beijing 100049, China}
\email{xyzhou@math.ac.cn}

\begin{abstract}
We give new characterizations of plurisubharmonic functions and
Griffiths positivity of holomorphic vector bundles with singular
Finsler metrics. As applications, we present a different method to
prove plurisubharmonic variation of generalized Bergman kernel
metrics and Griffiths positivity of the direct images of twisted
relative canonical bundles associated to holomorphic families of
K\"ahler manifolds.
\end{abstract}
\maketitle

\tableofcontents

\section{Introduction}
The aim of the present paper is to give a new characterization of
plurisubharmonic (p.s.h. for abbreviation) functions and a new
characterization of Griffiths positivity of holomorphic vector
bundles with singular Finsler metrics, and present a different
method to discuss plurisubharmonic variation of generalized Bergman
kernel metrics and Griffiths positivity of the direct image of the
twisted relative canonical bundle associated to a holomorphic family
of complex K\"ahler manifolds, by using the characterizations. The
work is inspired by Demailly's method to the regularization of
plurisubharmonic functions \cite{Dem92} and Berndtsson's proof of
the integral form of the Kiselman's minimum principle for
p.s.h. functions \cite{K78}, \cite{Bob98}.\\

To explain our observations, let us first take a look at Demailly's
method to the regularization of p.s.h. functions. Let $\varphi$ be a
p.s.h function on a bounded pseudoconvex domain $D\subset\mc^n$. Let
$K_m$ be the weighted Bergman kernel of $H^2(D, e^{-m\varphi})$, the
Hilbert space of holomorphic functions on $D$ which are
square-integrable with respect to the weight $e^{-m\varphi}$.
Applying the Ohsawa-Takegoshi extension theorem \cite{OT1}, Demailly
showed that $\frac{1}{m}\log K_m$ converges (in certain sense) to
$\varphi$ on $D$ as $m\rightarrow \infty$, and then got a
regularization of the original p.s.h function $\varphi$.\\

The first observation in this paper is a new characterization of p.s.h. functions
based on certain $L^p$ extension property.

We start from an upper semi-continuous function $\varphi$ on $D$
which is not assumed to be p.s.h. at the beginning, but we assume
that the Ohsawa-Takegoshi extension theorem also holds on $D$ with
the weight $e^{-m\varphi}$ for all $m\geq 1$. Then from the argument
of Demailly,  the convergence for $\frac{1}{m}\log K_m$ to $\varphi$
is also valid. On the other hand, it is clear that the logarithm of
the Bergman kernels $\frac{1}{m}\log K_m$ are always
plurisubharmonic. Therefore at the end one can see that $\varphi$ is
plurisubharmonic. This implies that a rough converse of the
Ohsawa-Takegoshi extension theorem holds. Precisely we have the
following slightly stronger result.

\begin{thm}\label{thm:char of p.s.h ftn,intr}
Let $\varphi:D\rightarrow [-\infty,+\infty)$ be an upper
semicontinuous function on $D\subset \mc^n$ that is not identically
$-\infty$. Let $p>0$ be a fixed constant. If for any $z_0\in D$ with
$\varphi(z_0)>-\infty$ and any $m>0$, there is $f\in\mathcal O(D)$
such that $f(z_0)=1$ and
$$\int_D|f|^{p}e^{-m\varphi}\leq C_me^{-m\varphi(z_0)},$$
where $C_m$ are constants independent of $z_0$ and satisfying the
growth condition $\lim\limits_{m\ra\infty}\frac{1}{m}\log C_m=0,$
then $\varphi$ is plurisubharmonic.
\end{thm}

\begin{rem}
It is worth to mention that Berndtsson proved that a continuous function $\varphi$ on a planar domain is subharmonic if
$e^{-m\varphi}$ can be used as a weight for H\"ormander's $L^2$-estimate for $\bar\partial$ \cite{Bob98}.
For the argument in \cite{Bob98}, dimension-one condition and continuity for $\varphi$ are necessary assumptions.
It seems interesting to generalize Berndtsson's result to higher dimensions and upper semi-continuous functions.
\end{rem}

Our second main observation is relating Theorem \ref{thm:char of p.s.h ftn,intr}
with positivity of direct image sheaves of twisted relative canonical boundles,
which is also inspired by Demailly's regularization of p.s.h. functions.
This observation comes from a new geometric interpretation of Theorem \ref{thm:char of p.s.h ftn,intr} as follows.

Let $D$ and $\varphi$ be as in Theorem \ref{thm:char of p.s.h ftn,intr}.
We view $e^{-\varphi}$ as a (singular) hermitian metric on the trivial line bundle $L=D\times \mc$.
Let $\pi=Id:D\ra D'=D$ be the trivial fibration with each fiber being one single point.
It is obvious that the associated twisted relative canonical bundle $K_{D/D'}\otimes L$ is isomorphic to $L$
and its direct image $L':=\pi_*(K_{D/D'}\otimes L)$ is also canonically isomorphic to $L$.
The Hodge-type metric on $L'$, denoted by $e^{-\varphi'}$, is given by integration along fibers.
In the case under discussion, since all fibers are single points,
the Hodge-type metric on $L'$ is given by valuations of the norms of the relative sections.
It is obvious that $\varphi'=\varphi$.
Therefore Theorem 1.1 implies that:
if $(mL,e^{-m\varphi})$ satisfies the Ohsawa-Takegoshi extension theorem for all $m\geq 1$,
then the direct image $\pi_*(K_{D/D'}\otimes L)$ is positively curved with respect to the Hodge-type metric.
This simple observation leads to the expectation that certain extension property should imply positivity of direct image sheaves.

In connection to this direction,
we give a characterization of Griffiths positivity of holomorphic vector bundles,
as a generalization of Theorem \ref{thm:char of p.s.h ftn,intr}.
We first introduce the notion of \emph{multiple $L^p$-extension property} for holomorphic vector bundles
with singular Finsler metrics.

\begin{defn}[Multiple $L^p$-extension property]\label{def-intr:multiple extension prop}
Let $(E,h)$ be a holomorphic vector bundle over a bounded domain
$D\subset\mc^n$ equipped with a singular Finsler metric $h$. Let
$p>0$ be a fixed constant. Assume that for any $z\in D$, any nonzero element $a\in E_{z}$
with finite norm $|a|$, and any $m\geq 1$, there is a holomorphic
section $f_m$ of $E^{\otimes m}$ on $D$ such that $f_m(z)=a^{\otimes
m}$ and satisfies the following estimate:
$$\int_D|f_m|^{p}\leq C_m|a^{\otimes{m}}|^{p}=C_m |a|^{mp},$$
where $C_m$ are constants independent of $z$ and satisfying the
growth condition $\frac{1}{m}\log C_m\ra 0$ as $m\ra\infty$. Then
$(E, h)$ is said to have multiple $L^p$-extension property.
\end{defn}

The following theorem says that multiple $L^p$-extension property for some $p>0$ implies Griffiths positivity.

\begin{thm}\label{thm-intr:cha of positive bundle}
Let $(E,h)$ be a holomorphic vector bundle over a bounded domain
$D\subset\mc^n$ equipped with a singular Finsler metric $h$, such
that the norm of any local holomorphic section of $E^*$ is upper
semicontinuous. If $(E,h)$ has multiple $L^p$-extension property for
some $p>0$, then $(E,h)$ is positively curved in the sense of
Griffiths, namely $\log|u|$ is plurisubharmonic for any local
holomorphic section $u$ of $E^*$.
\end{thm}

The reader is referred to  \S \ref{sec:metric on sheaf} for the
definitions of singular Finsler metrics and dual Finsler metrics,
and \S \ref{subsec:char. of positive v.b.} for the definition of
Finsler metrics on tensor products of vector bundles.

\begin{rem}
Some remarks about Theorem \ref{thm-intr:cha of positive bundle}:
\begin{itemize}
\item Although this theorem is stated and proved for vector bundles of finite rank, the same argument also works for holomorphic vector bundles of infinite rank;
\item In general a holomorphic vector bundle has certain positivity property if it has a lot of holomorphic sections.
 However the points here are different. Firstly the nature of this theorem is completely local,
and secondly our aim here is not to show that $E$ admits some positively curved metric,
instead our aim is to show that the metric $h$ itself is positively curved;
\item It seems that multiple $L^p$-extension property is stronger than Griffiths positivity,
and possibly it is more or less equivalent to Nakano positivity.
\end{itemize}
\end{rem}

The third observation in the present paper is introducing
fiber-product powers of holomorphic families. Motivated by
Demailly's method to regularization of p.s.h functions, it is
natural to expect that considering the tensor product powers  of a
line bundle twisting the relative canonical bundle may enable us to
prove the positivity of the direct image of the twisted relative
canonical bundles. However, this is not powerful enough when we
study families of complex manifolds with positive dimensional fibers
.

In this paper, we consider fiber-product powers of holomorphic
families, as well as tensor product powers of the involved line
bundle. Recall that for a holomorphic map $p:X\ra Y$ between complex
manifolds, the $m$-th power of fiber product is defined to be
$$X\times_Y\cdots\times_Y X:=\{(x_1,\cdots, x_m)\in X^m; p(x_1)=\cdots=p(x_m)\}.$$
This observation is inspired by Berndtsson's work on minimum
principle for p.s.h functions \cite{Bob98}.\\

We now discuss applications of the above observations to
plurisubharmonic variation of Bergman kernels and Griffiths
positivity of direct image sheaves associated to holomorphic
families of complex manifolds. The main tool in the arguments is an
Ohsawa-Takegoshi type $L^2$ extension theorem of the following form:

\begin{thm}\label{thm-intr:Demailly OT}
Let $(X,\omega)$ be a weakly pseudoconvex K\"{a}hler manifold
and $L$ be a holomorphic line bundle over $X$ with a (singular) hermitian metric $h$.
Let $s:X\ra\mc^r$ be a holomorphic map such that $0\in\mc^r$ is not a critical value of $s$.
Assume that the curvature current of $(L,h)$ is semi-positive and $|s(x)|\leq M$ for some constant $M$.
Let $Y=s^{-1}(0)$ be the zero set of $s$.
Then for every holomorphic section $f$ of $K_X\otimes L$ over $Y$
such that  $\int_Y|f|^2|\Lambda^r(ds)|^{-2}dV_\omega<+\infty$,
there exists a holomorphic section $F$ of  $K_X\otimes L$ over $X$
such that $F|_Y=f$ and
\begin{align*}
\int_X|F|_L^2dV_{X,\omega}\leq
C_{r,M}\int_Y\frac{|f|_L^2}{|\Lambda^r(ds)|^2}dV_{Y,\omega}
\end{align*}
where $C_{r,M}$ is a constant depending only on $r$ and $M$.
\end{thm}

\begin{rem}\label{rem-intr:OT with singular metric}
For Theorem \ref{thm-intr:Demailly OT}, the case that $X$ is a
pseudoconvex domain is proved by Ohsawa and Takegoshi in \cite{OT1}.
A geometric presentation in the case that $X$ is K\"ahler and $h$ is
smooth  was given by Manivel \cite{Man93} and Demailly \cite{Dem00}.
Recently, after the works by Blocki (\cite{Bl13}) and Guan-Zhou
(\cite{GZh12}, \cite{GZh15}, \cite {GZh15d}), the optimal $L^2$
extension in the setting of pseudoconvex K\"{a}hler manifolds presented as
above was proved by Cao \cite{Cao141} and Zhou-Zhu \cite{ZZ} with an
optimal estimate of the constant $C_{r,M}$. In this paper, we also
need an $L^p(0<p<2)$ variant of Theorem \ref{rem-intr:OT with
singular metric} due to Berndtsson-P\u{a}un \cite{BP10}.
\end{rem}

In \cite{GZh15d} (see also \cite{GZh17}) Guan and Zhou observed a
connection between positivity of the twisted relative canonical
bundles and Ohsawa-Takegoshi type extension, by showing that
plurisubharmonic variation of the relative Bergman kernels proved in
\cite{Bob06} and \cite{BP08} and Griffiths positivity of the twisted
relative canonical bundles \cite{Bob09a} can be deduced from Theorem
\ref{thm-intr:Demailly OT} with optimal estimate. On the other hand,
Berndtsson and Lempert \cite{BL16} show that Theorem
\ref{thm-intr:Demailly OT} with optimal estimate in the case of
pseudoconvex domains can be deduced from Berndtsson's result on
positivity of direct image of the twisted relative canonical bundles \cite{Bob09a}.

Roughly speaking, by the works of Guan-Zhou and Berndtsson-Lempert,
the Ohsawa-Takegoshi type extension theorems with optimal estimate
and the positivity of the twisted relative canonical bundle are
equivalent. Recently, based on Guan-Zhou's above observation, Hacon,
Popa and Schnell \cite{HPS16}, and P\u{a}un and Takayama \cite{PT18},
and Zhou and Zhu \cite{ZZ18} showed that, for a projective and
K\"ahler family of manifolds, the positivity of direct images of the
twisted relative pluricanonical bundles can be deduced from Theorem
\ref{thm-intr:Demailly OT} with optimal estimate.

Note that the Ohsawa-Takegoshi extension theorem with optimal
estimate and the positivity of direct image sheaves are equivalent,
it seems to be interesting that the positivity of direct image
sheaves could be deduced from  Theorem \ref{thm-intr:Demailly OT}
without optimal estimate. We should remark that the curvature
positivity property involves the mean value inequality, which is an
exact inequality, while the estimate in Theorem
\ref{thm-intr:Demailly OT} contains a constant which is not exact.
\\

In this paper, we'll present a different approach to the Griffiths
positivity of direct image sheaves based on the new
characterizations of p.s.h. functions , together with Theorem \ref{thm-intr:Demailly OT}.
The key insight is Theorem \ref{thm:char of p.s.h ftn,intr}, which says
that the submean value inequality can be deduced from a multiple
extension property, where nonexact constants are allowed provided
they satisfy certain growth condition.

The first applications of the above observations is to plursubharmonic variation
of $m$-Bergman kernels associated to a family of pseudoconvex domains.


Let $\Omega\subset\mc^r_t\times\mc^n_z$ be a pseudoconvex domain
and let $p:\Omega\ra U:=p(\Omega)\subset\mc^r$ be the natural projection.
Let $\varphi(t,z)$ be a p.s.h. function on $\Omega$.
Let $\Omega_t:=p^{-1}(t)$ $(t\in U)$ be the fibers.
We also denote by $\varphi_t(z)=\varphi(t,z)$ the restriction of $\varphi$ on $\Omega_t$.
Let $m$ be a postive integer, and let $K_{m,t}(z)$ be the $m$-Bergman kernel on $\Omega_t$ with wight $e^{-\varphi_t}$
(see \S \ref{subsec:m-Bergman kernel} for definition).
When $m=1$, $K_{1,t}(z)$ is the ordinary relative weighted Bergman kernel.
As $t$ varies,  $K_{m,t}(z)$ gives a function on $\Omega$.

\begin{thm}\label{thm-intr:psh variation of m Bergman kernel}
The function $\log K_{m,t}(z)$ is plurisubharmonic on $\Omega$.
\end{thm}

The method in the proof of Theorem \ref{thm-intr:psh variation of m
Bergman kernel} can also be applied to the case of families of
compact K\"ahler manifolds (see Theorem \ref{thm:p.s.h variation
Berg compact} and Remark \ref{rem:about boundedness of wt}).\\

The second application of the observations is to prove positivity of certain direct image sheaves.
We first consider the case of families of pesudoconvex domains.

Let $U, D$ be bounded pseudoconvex domains in $\mc^r$ and $\mc^n$ respectively,
and let $\Omega=U\times D\subset\mc^r\times\mc^n$.
Let $\varphi$ be a p.s.h function on $\Omega$,
which is assumed to be bounded.
For $t\in U$, let $D_t=\{t\}\times D$ and $\varphi_t(z)=\varphi(t,z)$.
Let $E_t=H^2(D_t, e^{-\varphi_t})$ be the space of $L^2$ holomorphic functions on $D_t$ with respect to the weight $e^{-\varphi_t}$.
Then $E_t$ are Hilbert spaces with the natural inner product.
Since $\varphi$ is assumed to be bounded on $\Omega$, all $E_t$ for $t\in U$ are equal as vector spaces.
However, the inner products on $E_t$ depend on $t$ if $\varphi(t,z)$ is not constant with $t$.
So, under the natural projection,  $E=\coprod_{t\in\Delta}E_t$ is a trivial holomorphic vector bundle (of infinite rank) over $U$
with varying Hermitian metric.

In \cite{Bob09a}, Berndtsson proved that $E$ is semipositive in the
sense of Griffths, namely, for any local holomorphic section $\xi$
of the dual bundle $E^*$, the function $\log|\xi|$ is
plurisubharmonic (indeed Berndtsson proved a stronger result which
says that $E$ is semipositive in the sense of Nakano). The aim here
is to apply the above observations to present a new proof of
Berntdsson's result. The argument in \cite{Bob09a} involves taking
derivatives and hence $\varphi$ is assumed to be smooth up to the
boundary of $\Omega$. In the following Theorem
\ref{thm-intr:Positivity of hodge domain case}, smoothness for
$\varphi$ is not necessary.

\begin{thm}\label{thm-intr:Positivity of hodge domain case}
The vector bundle $E$ is semipositive in the sense of Griffths.
\end{thm}

\begin{rem}\label{rem-intr:OT=>optimal OT}
We will see that Theorem \ref{thm-intr:Positivity of hodge domain case}
can be deduced from Theorem \ref{thm-intr:Demailly OT} for pseudoconvex domains.
We have mentioned that Theorem \ref{thm-intr:Demailly OT} with optimal estimate for pseudocovnex domains
can be deduced from Theorem \ref{thm-intr:Positivity of hodge domain case} \cite{BL16}.
So logically one can think that Theorem \ref{thm-intr:Demailly OT} with optimal estimate
is a consequence of Theorem \ref{thm-intr:Demailly OT} without optimal estimate.
However, this is not a proper viewpoint since the work in \cite{BL16}
contains some essentially new input, which can be viewed as a new technique of localization.
On the other hand, it is natural to expect that Theorem \ref{thm-intr:Demailly OT} with optimal estimate
can be deduced from Theorem \ref{thm-intr:Demailly OT},
by the method of raising powers of complex manifolds that is discussed above.
To do this, it seems that one need to prove a modification of the $L^2$-existence theorem in \cite{Dem00}.
\end{rem}

We now consider the positivity of the direct image sheaf
of the twisted relative canonical bundle associated to a family of compact K\"ahler manifolds.
Let $X, Y$ be K\"ahler  manifolds of  dimension $r+n$ and $r$ respectively,
and let $p:X\rightarrow Y$  be  a proper holomorphic map.
For $y\in Y$ let $X_y=p^{-1}(y)$ , which is a compact submanifold of $X$ of dimension $n$ if $y$ is a regular value of $p$.
Let $L$ be a holomorphic line bundle over $X$, and $h$ be a singular  Hermitian metric on $L$,
whose curvature current is semi-positive.
Let $K_{X/Y}$ be the relative canonical bundle on $X$.
Let $\mathcal E_m=p_*(mK_{X/Y}\otimes L\otimes \mathcal I_m(h))$
be the direct image sheaves on $Y$ (see \S \ref{sec:Ohsawa-Takegoshi} for the definition of the idea sheaf $\mathcal I_m(h)$).
One can choose a proper analytic subset $A\subset Y$ such that:
\begin{itemize}
\item[(1)] $p$ is submersive over $Y\backslash A$,
\item[(2)] $\mathcal E_m$  is locally free on $X\backslash A$, and
\item[(3)] $E_{m,y}$  is naturally identified with
$H^0(X_y,mK_{X_y}\otimes L|_{X_y}\otimes \mathcal I_m(h)|_{X_y})$, for $y\in Y\backslash A$.
\end{itemize}
where $E_m$  is the vector bundle on $Y\backslash A$
associated to $\mathcal E_m$.
For $u\in E_{m,y}$,
the $m$-norm of $u$ is defined to be
$$H_m(u):=\|u\|_m=\left(\int_{X_y}|u|^{2/m}h^{1/m}\right)^{m/2}\leq +\infty.$$
Then $H_m$ is a Finsler metric on $E_m$.
In the case that $m=1$, we denote $\mathcal E_1, H_1$
by $\mathcal E, H$ respectively.
The following theorem says that $H$ is a positively curved singular Hermitian metric
in the coherent sheaf $\mathcal E$ (see Definition \ref{def:finsler on sheaf} for definition).

\begin{thm}\label{thm-intr:Positivity of hodge compact case}
$H$ is a positively curved singular Hermitian metric on $\mathcal E$.
\end{thm}

In the proof of Theorem \ref{thm-intr:Positivity of hodge compact case},
we first show that $H$ is a positively curved Hermitian metric on $p_*(K_{X/Y}\otimes L)$,
and deduce form which the positivity of $(\mathcal E, H)$ by Oka-Grauert principle.

By the plurisubharonic variation of the $m$-Bergman kernel metrics
 and Theorem \ref{thm-intr:Positivity of hodge compact case},
one can see that the NS metric (see \S \ref{subsec:Bergman relative case} for definition)
on the direct image $p_*(kK_{X/Y}\otimes L\otimes \mathcal I_k(h))$
is positively curved in the sense of Griffiths.

Plurisubharmonic variation of Bergman kernels and positivity of
direct images of twisted relative canonical bundles have been
extensively studied in recent years by many authors (see e.g.
\cite{Bob06}\cite{Bob09a}\cite{BP08}\cite{BP10}\cite{GZh15d}\cite{DZZ17}\cite{PT18}\cite{Cao141}\cite{HPS16}\cite{ZZ}\cite{ZZ18}), in different settings and
different generality. The method in the present paper is quite
different from those in the cited works.

\begin{rem}
The same argument can be used to show that Theorem \ref{thm-intr:Positivity of hodge compact case}
still holds if $L$ is replaced by a holomorphic vector bundle with a Nakano semi-positive Hermitian metric (see also \cite{Liu-Yang14}).
\end{rem}

\begin{rem}
It is possible to prove that the metric $H_m$ (which is different from the NS metric defined in \S \ref{subsec:Bergman relative case}) is a
positively curved singular Finsler metric on the coherent sheaf $\mathcal E_m$ for all $m\geq 1$.
Considering their significance in birational classification of algebraic varieties (see e.g. \cite{Yau15}),
it seems that the study of the sheaves $\mathcal E_m$ is more important than that of their descendent objects - $m$-Bergman kernel metrics.
This topic will be discussed in a forthcoming work \cite{DWZZ18}.
\end{rem}


\subsection*{Acknowledgements}
The authors are partially supported by NSFC grants.

\section{$m$-Bergman kernel metric and Narasimhan-Simha metric}\label{sec:(relative)m-Bergman kernel}
In this section, we recall the so-called (relative) $m$-Bergman kernel metrics and Narasimhan-Simha metrics (NS metric for short) on
the twisted relative pluricanonical bundles associated to a family of complex manifolds.
We also prove a product property of (relative) $m$-Bergman kernels which will be frequently used in the rest of this paper.

\subsection{$m$-Bergman kernels associated to domains in $\mathbb{C}^n$}\label{subsec:m-Bergman kernel}
We first recall the definition of $m$-Bergman kernels on domains (see \cite{NZZ16} for more details).
Let $\Omega\subset \mathbb{C}^n$ be a domain.
Let $\varphi$ be an upper semicontinuous (u.s.c for short) function on $\Omega$.
Let $m\geq 1$ be an integer (indeed $m$ can be a real number in domain case).
The \emph{weighted $m$-Bergman space} is defined as
$$H_m(\Omega,\varphi):=\{u\in\mathcal{O}(\Omega): \|u\|_m:=\left(\int_\Omega |u|^{2/m}e^{-\varphi/m}\right)^{m/2}<+\infty \}.$$
Note that $\|u\|_m$ is not a norm if $m>1$ since the triangle inequality does not hold.
The associated $m$-Bergman kernel is defined as
\begin{align}\label{m Bergman kernel definition}
K_m(x):=\sup_{u\in H_m(\Omega,\varphi)\setminus \{0\}}\frac{|u(x)|^2}{\|u\|^2_m},
\end{align}
if $u(x)\neq 0$ for some $u\in H_m(\Omega,\varphi)$, and $K_m(x)=0$ if $u(x)=0$ for all $u\in H_m(\Omega,\varphi)$. Equivalently,  for any $x\in \Omega$ such that $K_m(x)\neq 0$,
\begin{align}\label{m Bergman kernel definition 2}
K_m(x)^{-1}=\inf\{\|u\|^2_m: u(x)=1, u\in H_m(\Omega,\varphi)\}.
\end{align}
\begin{rem}The $1$-Bergman kernel is  the  usual weighted Bergman kernel with weight $e^{-\varphi}$.
\end{rem}

The $m$-Bergman kernels have the  following product property.
\begin{prop}\label{prop:product property} Let $\Omega_1\subset \mathbb{C}^{n}, \Omega_2\subset \mathbb{C}^{s}$ be two domains, and $\varphi_1, \varphi_2$ be two  u.s.c.   functions on $\Omega_1$ and $\Omega_2$ respectively. Let $K_{m,1}, K_{m,2}, K_m$ be the $m$-Bergman kernels of $H_{m}(\Omega_1,\varphi_1)$, $H_m(\Omega_2,\varphi_2)$ and $H_m(\Omega_1\times \Omega_2,\varphi_1+\varphi_2)$ respectively. Then  for any point $(x_1,x_2)\in \Omega_1\times \Omega_2$,
\begin{align*}
K_m(x_1,x_2)=K_{m,1}(x_1)K_{m,2}(x_2).
\end{align*}
\end{prop}
\begin{proof}
It is obvious that $K_m(x_1,x_2)\geq K_{m,1}(x_1)K_{m,2}(x_2)$. It suffices to prove that $K_m(x_1,x_2)\leq K_{m,1}(x_1)K_{m,2}(x_2)$. If $K_m(x_1,x_2)=0$, it is trivial, since one of $K_m(x_1)$ and $K_{m}(x_2)$ equals zero.

Now we assume that $K_m(x_1,x_2)\neq 0$.  Let $u\in \mathcal{O}(\Omega_1\times \Omega_2)$, such that $u(x_1,x_2)=1$, and $u\in H_m(\Omega_1\times \Omega_2, \varphi_1+\varphi_2)$
Then from the definition of $m$-Bergman kernel, one can obtain that
\begin{align*}
\int_{\Omega_2}|u(z_1, z_2)|^{2/m}e^{-\varphi_2(z_2)/m}d\lambda_{z_2}\geq \frac{|u(z_1,x_2)|^{2/m}}{(K_{m,2}(x_2))^{1/m}}.
\end{align*}
By Fubini theorem and the definition of $m$-Bergman kernel,
\begin{align*}
&(\int_{\Omega_1\times \Omega_2}|u(z_1,z_2)|^{2/m}e^{-(\varphi_1(z_1)+\varphi_2(z_2))/m}d\lambda_{z_1}d\lambda_{z_2})^{m}\\
&\geq \frac{1}{K_{m,2}(x_2)}(\int_{\Omega_1}|u(z_1,x_2)|^{2/m}e^{-\varphi_1(z_1)/m}d\lambda_{z_1})^{m}\\
&\geq \frac{1}{K_{m,1}(x_1)}\cdot\frac{1}{K_{m,2}(x_2)}
\end{align*}
Taking infimum on all $u\in H_m(\Omega_1\times\Omega_2, \varphi_1+\varphi_2)$ with $u(x_1,x_2)=1$, we get
\begin{align*}
\frac{1}{K_{m}(x_1,x_2)}\geq  \frac{1}{K_{m,1}(x_1)}\cdot\frac{1}{K_{m,2}(x_2)},
\end{align*}
and hence $K_m(x_1,x_2)\leq K_{m,1}(x_1)K_{m,2}(x_2)$.
\end{proof}

\subsection{$m$-Bergman kernel metrics on complex manifolds}\label{subsec:Bergman manifold-case}
Let $X$ be a compact manifold of dimension $n$,
and let $L$ be a holomorphic line bundle on $X$ with a (singular) Hermitian metric $h=e^{-\varphi}$.
We assume that local weights $\varphi$ of $h$ take values in $[-\infty,+\infty)$ and are upper semicontinuous.
For a section $u\in H^0(X, mK_X\otimes L)$, we define its $m$-norm as
$$\|u\|_m:=\left(\int_X|u|^{2/m}e^{\varphi/m}\right)^{m/2}.$$
Let $H^0_m(X,mK_X\otimes L)=\{u\in H^0(X,mK_X\otimes L); \|u\|_m<\infty\}$.
If $H^0_m(X,mK_X\otimes L)\neq \{0\}$, it induces a metric on $mK_X\otimes L$ as follows.
For $x\in X$ such that $u(x)\neq 0$ for some $u\in H^0_m(X,mK_X\otimes L)$,
the evaluation map $H^0_m(X,mK_X\otimes L)\ra (mK_X\otimes L)|_x$ is surjective and hence induces a metric on $(mK_X\otimes L)|_x$
given by
$$|v|_m:=\inf\{\|u\|_m: u\in H^0_m(X,mK_X\otimes L), u(x)=v\},$$
$v\in (mK_X\otimes L)|_x$; for $x\in X$ such that $u(x)=0$ for all $u\in H^0_m(X,mK_X\otimes L)$,
the metric on $(mK_X\otimes L)|_x$ is defined to be $+\infty$ for all nonzero vectors.

The metric defined above is called the \emph{$m$-Bergman kernel metric} on $mK_X\otimes L$ and will be denoted by $B^m_{(L,h)}$.
The local weights of $B^m_{(L,h)}$ can be given as follows.
Assume $(U,z)$ is a local coordinate on $X$ and $e$ is a local frame of $L$ on $U$.
Then an element $u\in H^0(X,mK_X\otimes L)$ can be represented on $U$ as $u=\tilde u dz^{\otimes m}\otimes e$,
where $\tilde u$ is a holomorphic function on $U$.
Let $K_{(U,m)}(x)=\sup\{\frac{1}{\|u\|^2_m}; u\in H_m^0(X,mK_X\otimes L), \tilde u(x)=1\}, \ x\in U$
if $u(x)\neq 0$ for some $u\in H_m^0(X,mK_X\otimes L)$, and set $K_{(U,m)}(x)$ otherwise.
Then $\ln K_{(U,m)}$ is  a local weight of $B_{(L,h)}$ on $U$ with respect to the frame $dz^{\otimes m}\otimes e$;
namely, with respect to the $m$-Bergman metric,
$$|dz^{\otimes m}\otimes e|^2=e^{-\ln K_{(U,n)}}=\frac{1}{K_{(U,m)}}.$$

The relation between the $m$-Bergman kernel defined in the previous subsection and the $m$-bergman kernel metric is as follows.
For a domain $\Omega\subset\mc^n$ and a p.s.h function $\varphi$ on $\Omega$,
viewing $e^{-\varphi}$ as a metric on the trivial line bundle $L=\Omega\times \mc$ over $\Omega$,
then the $m$-Bergman kernel metric on $mK_\Omega\otimes L\equiv \Omega\times \mc$ is given by $\frac{1}{K_m}$,
where $K_m$ is the $m$-Bergman kernel on $\Omega$ with weight $\varphi$.
Therefore, the $m$-Bergman kernel metric can be seen as an intrinsic definition of the $m$-Bergman kernel.


Similar to Proposition \ref{prop:product property}, we have the following product property for $m$-Bergman kernel metrics.

\begin{prop}\label{prop:product property-manifold}
Let $X_1$ and $X_2$ be two complex manifolds,
and $L_1\rightarrow X_1$ and $L_2\rightarrow X_2$ be two holomorphic line bundles over $X_1$ and $X_2$ respectively.
Let $h_1$ and $h_2$ be two Hermitian metrics on $L_1$ and $L_2$ respectively.
Let $p_i:X_1\times X_2\rightarrow X_i$ $(i=1,2)$ be the natural projections.
Let $L$ be the induced holomorphic line bundle $p_1^*L_1\otimes p_2^*L_2$ over $ X_1\times X_2$ with the Hermitian metric $h=p_1^*h_2\cdot p_2^*h_2$.
Let $B^m_{(L_1,h_1)}$, $B^m_{(L_2,h_2)}$ and $B^m_{(L,h)}$ be the corresponding $m$-Bergman kernel metrics
on $mK_{X_1}\otimes L_1, mK_{X_2}\otimes L_2$ and $mK_{X}\otimes L$ respectively.
Then
$$B^m_{(L,h)}=B^m_{(L_1,h_1)}\cdot B^m_{(L_2,h_2)}.$$
\end{prop}

The statement in Proposition \ref{prop:product property-manifold} is understood as follows.
In a canonical way, we can identify $K_X$ with $p_1^*K_{X_1}\otimes p^*_2K_{X_2}$,
and therefore identify $mK_X\otimes L$ with $(mK_{X_1}\otimes L_1)\otimes (mK_{X_2}\otimes L_2)$.
Then Proposition \ref{prop:product property-manifold} says that the $m$-Bergman kernel metric
$B^m_{(L,h)}$ is induced by the $m$-Bergman kernel metrics $B^m_{(L_1,h_1)}, B^m_{(L_2,h_2)}$
on $(mK_{X_1}\otimes L_1), (mK_{X_2}\otimes L_2)$.

\begin{proof}
The proof of Proposition \ref{prop:product property-manifold} is similar
to that of Proposition \ref{prop:product property}.

For simplicity, we denote $B^m_{(L_1, h_1)}, B^m_{(L_2, h_2)}, B^m_{(L, h)}$ by $B_1, B_2, B$ respectively.

By definition, it is obvious that $B\leq B_1\otimes B_2$.
We now prove  that $B\geq B_1\otimes B_2$.

Fix an arbitrary $(x_1,x_2)\in X_1\times X_2$.
If $B(x_1, x_2)=+\infty$, there is nothing to prove.
We assume that $B(x_1, x_2)\neq +\infty$.
Let $a\in (mK_{X_1}\otimes L_1)|_{x_1}, b\in (mK_{X_2}\otimes L_2)|_{x_2}$ with $a, b\neq 0$.
Let $u\in H^0_m(X_1\times X_2, mK_{X_1\times X_2}\otimes L)$ such that $u(x_1,x_2)=a\otimes b$.
From the definition of the Bergman kernel,
\begin{align*}
\int_{X_2} |u(z,w)|^{2/m}h^{1/m}_2 \geq |u(z,x_2)|^{2/m}B^{1/m}_2(b).
\end{align*}
By Fubini theorem and the definition of $m$-Bergman kernel,
\begin{align*}
&\left(\int_{X_1\times X_2}|u(z,w)|^{2/m}h^{1/m}_1 h^{1/m}_2\right)^m\\
&\geq B_2(b)\left(\int_{X_1}|u(z,x_2)|^{2/m}h^{1/m}_2\right)^m\\
&\geq B_1(a)B_2(b)=(B_1\cdot B_2)(a\otimes b).
\end{align*}
Taking infimum on all such $u$, we complete the proof.
\end{proof}

\subsection{Relative Bergman kernel metrics and Narasimhan-Simha metrics}\label{subsec:Bergman relative case}
Let $X$ be a compact complex manifold of dimension $n$,
and let $L$ be a holomorphic line bundle on $X$ with a (singular) Hermitian metric $h=e^{-\varphi}$.
We assume that local weights $\varphi$ of $h$ take values in $[-\infty,+\infty)$ and are upper semicontinuous.
In \S \ref{subsec:Bergman manifold-case} we have defined the $m$-Bergman kernel metric $B^m_{(L,h)}$ on $mK_X\otimes L$.
It is clear that $h_m:=(B^m_{(L,h)})^{\frac{m-1}{m}}h^{\frac{1}{m}}$ defines a metric on $L_m:=(m-1)K_X\otimes L$,
which is called the \emph{Narasimhan-Simha metric} (NS metric for short) on $L_m$.
Then the Bergman kernel metric on $H^0(X, K_X\otimes L_m)=H^0(X,mK_X\otimes L)$ induced from $h_m$ is also called the \emph{Narasimhan-Simha metric} on $H^0(X,mK_X\otimes L)$.

Let $X$ and $Y$ be complex manifolds of dimension $n+s$ and $s$ respectively.
Let $p:X\rightarrow Y$ be a  holomrophic surjective map. Denote by $X_y:=p^{-1}(y)$.
Note that  if $p$ is a submersion,
then all the fibers $X_y$ are complex submanifolds of the same dimension.
Let $K_X$, $K_Y$  be the cannonical bundles on $X$, $Y$ respectively,
 and $K_{X/Y}:=K_X-p^*K_Y$  be the relative canonical bundle on $X$.
 Let $(L,h)\rightarrow X$ be a holomorphic  line bundle over $X$ with a (singular) Hermitian mertic $h$,
 whose weights are locally integrable functions.

Firstly, we assume that $p$ is a submersion. Let  $t=(t_1,\cdots, t_s)$ be  local coordinates on an open set  $U\subset Y$. Then $dt:=dt_1\wedge \cdots\wedge dt_s$ is a trivialization of  $K_Y$ on $U$, and gives a natural map from $(n,0)$-forms $u$ on the fibers $X_y$ to sections $\widetilde{u}$ of $K_X$ over $X_y$ for $y\in U$ by
\begin{align*}
\widetilde{u}=u\wedge p^*(dt).
\end{align*}
Conversely, given a local section $\widetilde{u}$ of $K_X$, we can write $\widetilde{u}=u\wedge dt$ locally. The restriction of $u$ to fibers is then uniquely defined and thus defines a  section of  $K_{X_y}$. From the correspondence between $u$ and $\widetilde{u}$, we  get an isomorphism
\begin{align*}
K_{X_y}&\rightarrow K_{X/Y}|_{X_y}\\
u&\mapsto  u\wedge\frac{p^*(dt)}{dt}.
\end{align*}
It is worth to mention that through the correspondence between $u$ and $\widetilde{u}$ depends on the choice of $dt$, but the above isomorphism is independent of the choice of $dt$. Similarly, we have the following isomorphism
\begin{align*}
mK_{X_y}&\rightarrow mK_{X/Y}|_{X_y}\\
u&\mapsto u\wedge\frac{(p^*(dt))^{\otimes m}}{(dt)^{\otimes m}},
\end{align*}
which further induces a canonical isomorphism between $mK_{X_y}\times L|_{X_y}$ and $(mK_{X/Y}\otimes L)|_{X_y}$.

Recall that we have defined the $m$-Bergman kernel metric on $mK_{X_y}\times L|_{X_y}$,
which can also be viewed an Hermitian metric on $(mK_{X/Y}\otimes L)|_{X_y}$, according to the above isomorphism.
As $y$ varies in $Y$, we get a (singular) Hermitian metric on $mK_{X/Y}\otimes L$,
which will be called the relative $m$-Bergman kernel metric on $mK_{X/Y}\otimes L$.
The relative NS metric on $(m-1)K_{X/Y}\otimes L$ is defined in the same way
from the NS metric on $(m-1)K_{X_y}\times L|_{X_y}$.

If $p:X\ra Y$ is not a submersion,
the relative $m$-Bergman kernel metric given above is only defined on $(mK_{X/Y}\otimes L)|_{p^{-1}(U)}$
for some Zariski open set $U$ in $Y$.
One of the main aims of the paper is to show that, for some cases, the $m$-Bergman kernel metric on $(mK_{X/Y}\otimes L)|_{p^{-1}(U)}$
is positively curved and can be extended to a positively curved metric on $mK_{X/Y}\otimes L$.

\section{Singular Finsler metrics on coherent analytic sheaves}\label{sec:metric on sheaf}
In this section,
we recall the notions of singular Finsler metrics on holomorphic vector bundles
and give a definition of positively curved singular Finsler metrics on coherent analytic sheaves.

\begin{defn}\label{def:finsler on v.b.}
Let $E\rightarrow X$ be a holomorphic vector bundle over a complex manifold $X$. A (singular) Finsler metric $h$ on $E$
is a function $h:E\ra [0,+\infty]$, such that $|v|^2_h:=h(cv)=|c|^2h(v)$ for any $v\in E$ and $c\in\mc$.
\end{defn}
In the above definition, we do not assume any regularity property of a singular Finsler metric.
Only when considering Griffiths positivity certain regularity is required,
as shown in the following Definition \ref{def:finsler on v.b.}.

\begin{defn}\label{def:dual finsler metric}
For a singular Finsler metric $h$ on $E$,
its dual Finsler metric $h^*$ on the dual bundle $E^*$ of $E$ is defined as follows.
For $f\in E^*_x$, the fiber of $E^*$ at $x\in X$,
$|f|_{h^*}$ is defined to be $0$ if $|v|_h=+\infty$ for all nonzero $v\in E_x$;
otherwise,
$$|f|_{h^*}:=\sup\{|f(v)|; v\in E_x, |v|_h\leq 1\}\leq +\infty.$$
\end{defn}

\begin{defn}\label{def:positivity of finsler}
Let $E\rightarrow X$ be a holomorphic vector bundle over a complex manifold $X$.
A singular Finsler metric $h$ on $E$ is called negatively curved (in the sense of Griffiths)
if for any local holomorphic section $s$ of $E$ the function $\log|s|^2_h$ is plurisubharmonic,
and is called positively curved (in the sense of Griffiths) if its dual metric $h^*$ on $E^*$ is negatively curved.
\end{defn}


As far as our knowledge, there have not been natural definition of singular Finsler metric on a coherent analytic sheaf.
In the present paper, we will propose a definition of  positively curved Finsler metrics on coherent analytic sheaves.
Let $\mathcal F$ be a coherent analytic sheaf on $X$,
it is well known that $\mathcal F$ is locally free on some Zariski open subset $U$ of $X$.
On $U$, we will identify $\mathcal F$ with the vector bundle associated to it.

\begin{defn}\label{def:finsler on sheaf}
Let $\mathcal{F}$ be a coherent analytic sheaf on a complex manifold $X$.
Let $Z\subset X$ be an analytic subset of $X$ such that $\mathcal{F}|_{X\setminus Z}$ is locally free.
A positively curved singular Finsler metric $h$ on $\mathcal{F}$ is a singular Finsler metric on the holomorphic vector bundle $\mathcal{F}|_{X\setminus Z}$,
such that for any local holomorphic section $g$ of the dual sheaf $\mathcal{F}^*$ on an open set $U\subset X$,
the function $\log|g|_{h^*}$ is p.s.h.  on  $U\setminus Z$, and can be extended to a p.s.h. function on $U$.
\end{defn}

\begin{rem}
Suppose that $\log|g|_{h^*}$  is p.s.h. on $U\setminus Z$.
It is well-known that if  codim$_{\mathbb{C}}(Z)\geq 2$ or $\log|g|_{H^*}$ is locally bounded above near $Z$,
then $\log|g|_{h^*}$ extends across $Z$ to $U$ uniquely as a p.s.h function.
Definition \ref{def:finsler on sheaf} matches Definition \ref{def:finsler on v.b.}
and Definition \ref{def:positivity of finsler} if $\mathcal F$ is a vector bundle.
\end{rem}

\section{Extension theorems of Ohsawa-Takegoshi type} \label{sec:Ohsawa-Takegoshi}
In this section, we prepare some extension theorems of Ohsawa-Takegoshi type
which will be used as a central tool in this paper.

\begin{thm}\label{thm:Demailly OT}
Let $(X,\omega)$ be a weakly pseudoconvex K\"{a}hler manifold
and $L$ be a holomorphic line bundle over $X$ with a (singular) hermitian metric $h$.
Let $s:X\ra\mc^r$ be a holomorphic map such that $0\in\mc^r$ is not a critical value of $s$.
Assume that the curvature current of $(L,h)$ is semi-positive and $|s(x)|\leq M$ for some constant $M$.
Let $Y=s^{-1}(0)$ be the zero set of $s$.
Then for every holomorphic section $f$ of $K_X\otimes L$ over $Y$
such that  $\int_Y|f|^2|\Lambda^r(ds)|^{-2}dV_\omega<+\infty$,
there exists a holomorphic section $F$ of  $K_X\otimes L$ over $X$
such that $F|_Y=f$ and
\begin{align*}
\int_X|F|_L^2dV_{X,\omega}\leq C_{r,M}\int_Y\frac{|f|_L^2}{|\Lambda^r(ds)|^2}dV_{Y,\omega}.
\end{align*}
where $C_{r,M}$ is a constant depending only on $r$ and $M$.
\end{thm}

\begin{rem}\label{rem:OT with singular metric}
For Theorem \ref{thm-intr:Demailly OT}, the case that $X$ is a
pseudoconvex domain is proved by Ohsawa and Takegoshi in \cite{OT1}.
A geometric presentation in the case that $X$ is K\"ahler and $h$ is
smooth  was given by Manivel \cite{Man93} and Demailly \cite{Dem00}.
Recently, after the works by Blocki (\cite{Bl13}) and Guan-Zhou
(\cite{GZh12}, \cite{GZh15}, \cite {GZh15d}), the optimal $L^2$
extension in the setting of pseudoconvex K\"{a}hler manifolds presented as
above was proved by Cao \cite{Cao141} and Zhou-Zhu \cite{ZZ} with an
optimal estimate of the constant $C_{r,M}$.
\end{rem}

Combining Theorem \ref{thm:Demailly OT} and the  iteration method in \cite{BP10}, we ge the following

\begin{thm}[\cite{BP10}]\label{thm:Lm-extension on psc}
Let $\Omega\subset \mathbb{C}^{n+r}$ be a pseudoconvex domain,
and $p: \Omega\rightarrow p(\Omega)\subset \mathbb{C}^r$ be the natural projection.
For $y\in p(\Omega)$, we denote $\Omega_y:=p^{-1}(y)$ by $\Omega_y$.
Let $\varphi$  be a p.s.h function on $\Omega$.
Assume that $|y|\leq M$ for all $y\in p(\Omega)$.
Let $m\geq 1$ be an integer and $y_0\in p(\Omega)$ such that $\varphi$ is not identically $-\infty$ on any branch of $\Omega_{y_0}$.
Then for any holomorphic function $u$ on $\Omega_{y_0}$ such that
\begin{align*}
\int_{\Omega_{y_0}}|u|^{2/m}e^{-\varphi}<+\infty,
\end{align*}
there exists a holomorphic function $U$ on $\Omega$ such that $U|_{\Omega_{y_0}}=u$ and
\begin{align*}
\int_{\Omega}|U|^{2/m}e^{-\varphi}\leq C_{r,M}\int_{\Omega_{y_0}}|u|^{2/m}e^{-\varphi},
\end{align*}
where $C_{r,M}$ is the constant as in Theorem \ref{thm:Demailly OT}.
\end{thm}

\begin{proof}
We follow the idea as in the proof of  Theorem  in \cite{BP10}.
By standard regularization argument,  without loss of generality,
one may assume that $\varphi\in Psh(\overline{\Omega})\cap\mathcal{C}^\infty(\overline{\Omega})$,
$u$ is holomorphic on some neighborhood of $\overline{\Omega}_{y_0}$.
We can also assume
\begin{align*}
\int_{\Omega_{y_0}}|u|^{2/m}e^{-\varphi}d\lambda=1.
\end{align*}
From the pseudoconvexity of $\Omega$,
one can find a holomorphic function $F_1\in \mathcal{O}(\Omega)$ with $F_1|_{{\Omega}_{y_0}}=u(x)$.
Replacing $\Omega$ by a relatively compact domain in it, we can assume
\begin{align*}
\int_\Omega|F_1|^{2/m}e^{-\varphi}d\lambda\leq A<+\infty.
\end{align*}
Let $\varphi_1=\varphi+(1-1/m)\log|F_1|^2$,
by Theorem \ref{thm:Demailly OT} with the weight $\varphi_1$,
there is a new extension of $F_2$ of $u$ satisfying
\begin{align*}
\int_\Omega\frac{|F_2|^2}{|F_1|^{2-2/m}}e^{-\varphi}d\lambda\leq C_{r,M}\int_{\Omega_{y_0}}\frac{|u|^{2}}{|F_1|^{2-2/m}}e^{-\varphi}d\lambda_V=C_{r,M}.
\end{align*}
By H\"{o}lder's inequality,
\begin{align*}
\int_\Omega|F_2|^{2/m}e^{-\varphi}d\lambda&\leq \Big(\int_\Omega\frac{|F_2|^2}{|F_1|^{2-2/m}}e^{-\varphi}d\lambda\Big)^{1/m}\Big(\int_\Omega|F_1|^{2/m}e^{-\varphi}d\lambda\Big)^{(m-1)/m}\\
&\leq C_{r,M}^{1/m}A^{(m-1)/m}=A(C_{r,M}/A)^{1/m}=:A_1.
\end{align*}
We can assume $A>C_{r,M}$, then $A_1<A$.
Repeating the same argument with $F_1$ replaced by $F_2$, etc,
we get a decreasing sequence of constants $A_k$, such that
\begin{align*}
A_{k+1}=A_k(C_{r,M}/A_k)^{1/m}
\end{align*}
for $k\geq 1$.
It is easy to see that $A_k$ tends to $C_{r,M}$.
Taking limit, we obtain a holomorphic function $U$ on $\Omega$ extending $u$, such that
\begin{align*}
\int_\Omega |U(x)|^{2/m}e^{-\varphi}d\lambda\leq C_{r,M}.
\end{align*}
\end{proof}

Theorem \ref{thm:Lm-extension on psc} can be extended to a family of compact K\"ahler manifolds.
Let $B\subset \mc^r$ be the unit ball and let $X$ be a K\"ahler manifold of dimension $n+r$.
Let $p:X\ra B$ be a holomorphic proper submersion.
For $t\in B$, denote by $X_t$ the fiber $p^{-1}(t)$.
Let $L$ be a holomorphic line bundle on $X$ with a (singular) Hermitian metric $h$
whose curvature current is positive.

Let $k>0$ be a fixed integer.
The multiplier ideal sheaf $\mathcal I_k(h)\subset \mathcal O_X$ is defined as follows.
If $\varphi$ is a local weight of $h$ on some open set $U\subset X$,
then the germ of $\mathcal I_k(h)$ at a point $p\in U$ consists of the germs
of holomorphic functions $f$ at $p$ such that $|f|^{2/k}e^{-\varphi/k}$ is integrable at $p$.
It is known that $\mathcal I(h^{1/k})$ is a coherent analytic sheaf on $X$ \cite{Cao141}.

Let $\mathcal E=p_*(kK_{X/B}\otimes L\otimes \mathcal I_k(h))$ be the direct image sheaf on $B$.
For any open subset $U$ of $B$ containing the origin 0 and any $s\in H^0(U,\mathcal E|_U))$,
the restriction of $s$ on $X_0$, denoted by $s|_{X_0}$, gives a section in $H^0(X_0, kK_{X_0}\otimes \mathcal I_k(h)|_{X_0})$.
By Cartan's Theorem B, there exists a global section $\tilde s$ of $\mathcal E$ on $B$ such that $\tilde s|_{X_0}=s|_{X_0}$.
For For $u\in H^0(X_0, kK_{X_0}\otimes\mathcal I_k(h)|_{X_0})$, as in \S \ref{subsec:Bergman manifold-case},
the $k$-norm of $u$ is defined to be
$$\|u\|_k=\left(\int_{X_0}|u|^{2/k}h^{1/k}\right)^{k/2}\leq +\infty.$$


Modifying the argument in the proof of Theorem \ref{thm:Lm-extension on psc},
we can prove the following

\begin{thm}\label{thm:Lm-extension projective}
Let $u\in H^0(X_0, kK_{X_0}\otimes\mathcal I_k(h)|_{X_0})$ with $\|u\|_k<\infty$.
Assume there exist an open subset $U$ containing the origin and $s_0\in H^0(U,\mathcal E|_U)$ such that $s_0|_{X_0}=u$,
then there exists $s\in H^0(X,kK_X\otimes L\otimes \mathcal I_k(h))=H^0(B, kK_B\otimes \mathcal E)$
such that $s|_{X_0}=u\wedge dt^{\otimes k}$ and
$$\int_X |s|^{2/k}h^{1/k}\leq C_{r,1}\int_{X_0}|u|^{2/k}h^{1/k},$$
where $t=(t_1,\cdots, t_r)$ is the standard coordinate on $B$ and $dt=dt_1\wedge\cdots\wedge dt_r$,
and $C_{r,1}$ is the constant as in Theorem \ref{thm:Lm-extension on psc}.
\end{thm}
\begin{proof}
We assume that $\|u\|_k=1$.
As explained above, there is a $\tilde s_1\in H^0(B,\mathcal E)$ such that $\tilde s_1|_{X_0}=u$.
Let $s_1=\tilde s_1\wedge dt$.
By replacing $B$ by a relatively smaller ball,
we can assume that  $\int_X |s_1|^{2/k}h^{1/k}\leq A<\infty$ for some constant $A$.
The section $s_1$ induces a singular Hermitian metric
$$h_1=\left(\frac{1}{|s_1|^2}\right)^{\frac{k-1}{k}}h^{\frac{1}{k}}$$
on $(k-1)K_X\otimes L$, whose curvature current is positive.
By Theorem \ref{thm:Demailly OT}, there is a section $s_2\in H^0(X, K_X\otimes ((k-1)K_X\otimes L))$
such that
$$\int_X\frac{|s_2|^2}{|s_1|^{2-2/k}}h^{1/k}\leq C_{r,1}\int_{X_0}\frac{|u|^{2}}{|s_1|^{2-2/k}}h^{1/k}=C_{r,1}.$$
By H\"{o}lder's inequality,
\begin{align*}
\int_X |s_2|^{2/k}h^{1/k}&\leq \Big(\int_X\frac{|s_2|^2}{|s_1|^{2-2/k}}h^{1/k}\Big)^{1/k}\Big(\int_X |s_1|^{2/k}h^{1/k}\Big)^{(k-1)/k}\\
&\leq C_{r,1}^{1/k}A^{(k-1)/k}=A(C_{r,1}/A)^{1/k}=:A_1.
\end{align*}
We can assume $A>C_{r,1}$, then $A_1<A$.
Repeating the same argument with $s_1$ replaced by $s_2$, etc,
we get a decreasing sequence of constants $A_k$, such that
\begin{align*}
A_{k+1}=A_k(C_{r,M}/A_k)^{1/m}
\end{align*}
for $k\geq 1$.
It is easy to see that $A_k$ tends to $C_{r,M}$.
Taking limit, we obtain a section $s$ that satisfies the condition in the theorem.
\end{proof}

\section{Regularity of Bergman kernel metrics and Hodge-type metrics}\label{sec:regularity}
The aim of this section is to show certain continuity of relative $m$-Bergman kernel metrics and Hodge-type metrics on direct image sheaves.

\subsection{For families of compact K\"ahler manifolds}\label{sebsec:beg regu proj family}
Let $X, Y$ be K\"ahler  manifolds of  dimension $m+n$ and $m$ respectively,
let $p:X\rightarrow Y$  be a proper holomorphic submersion.
Let $L$ be a holomorphic line bundle over $X$, and $h$ be a singular  Hermitian metric on $L$,
whose curvature current is semi-positive. Let $K_{X/Y}$ be the relative canonical bundle on $X$.

Let $\mathcal E_k=p_*(kK_{X/Y}\otimes L\otimes \mathcal I_k(h))$ be the direct image sheaf on $Y$.
By Grauert's theorem, $\mathcal E_k$ is a coherent analytic sheaf on $Y$.
We assume that $\mathcal E_k$ is locally free,
then it is the sheaf of holomorphic sections of a holomorphic vector bundle,
which will be denoted by $E_k$.
For any $y\in Y$,
we can identify the fiber $E_{k,y}$ of $E_k$ at $y$ with $H^0(X_y, (kK_{X/Y}\otimes L\otimes \mathcal I_k(h))|_{X_y})\subset H^0(X_y, kK_{X_y}\otimes L|_{X_y})$.
For $u\in E_{k,y}$, as in \S \ref{subsec:Bergman manifold-case},
the $k$-norm of $u$ is defined to be
$$H_y(u):=\|u\|_k=\left(\int_{X_y}|u|^{2/k}h^{1/k}\right)^{k/2}\leq +\infty.$$
Note that here we view $u$ as an element in $H^0(X_y, kK_{X_y}\otimes L|_{X_y})$.
Then $H$ is a Finsler metric on $E_k$.
It is clear that $H$ is locally bounded below by positive constants.
The following proposition shows that $H$ is lower semicontinuous.

\begin{prop}[\cite{HPS16}]\label{prop:Hodge metric lower semi-continuous:compact}
Let $s$ be a holomorphic section of $E_k$.
The function $|s|_k(y):=\|s(y)\|_k:Y\rightarrow [0,+\infty]$ is lower semi-continuous.
\end{prop}
\begin{proof}
We present the proof given in \cite{HPS16}.
Without loss of generality, we assume that $Y=B$, which is  the unit ball in $\mathbb{C}^m$
and prove that $|s|_k$ is lower semicontinuous at the origin 0.
Denote by $(t_1,\cdots, t_m)$ the standard coordinate system on $B$,
then the canonical bundle $K_B$ is trivialized by the global section $dt=dt_1\wedge \cdots \wedge dt_m$,
and the volume form on $B$ is
\begin{align*}
d\mu=c_mdt\wedge d\overline{t}.
\end{align*}
Denote by
\begin{align*}
\beta=s\wedge (dt)^{\otimes k}\in H^0(B,kK_B\otimes E_k)\simeq H^0(X,kK_X\otimes L\otimes \mathcal{I}_k(h)).
\end{align*}
Since $p:X\rightarrow B$ is a submersion, Ehresmann's fibration theorem
shows that $X$ is diffeomorphic to the product $ B\times X_0$.
Choosing a K\"{a}hler metric $\omega_0$ on $X_0$,
we can write
\begin{align*}
|\beta|^{2/k}h^{1/k}=F\cdot d\mu\wedge\frac{\omega_0^{n}}{n!}
\end{align*}
where $F:B\times X_0\rightarrow [0,+\infty]$ is lower semi-continous and locally integrable;
the reason is that the local weights for $(L,h)$ are upper semi-continuous.
At every point $y\in B$, we have that
\begin{align*}
|s(y)|_{k,y}=\Big(\int_{X_0}F(y,-)\frac{\omega_0^n}{n!}\Big)^{k/2}
\end{align*}
By Fubini's theorem,
the function $|s|_k$ is measurable function on $B$.
Moreover, since $F$ is locally integrable and $X_0$ is compact,
$\|s(y)\|_{k}<+\infty$ for almost every $y\in B$.

We now need to show that
\begin{align}\label{lower-compact-ineq}
|s(0)|_{k}\leq \liminf\limits_{j\rightarrow +\infty}|s(y_k)|_{k}
\end{align}
holds for every sequence $y_1,y_2,\cdots\in B$  which converges to the origin.
By the lower semi-continuity of $F$ and Fatou's lemma, we obtain
\begin{align*}
\int_{X_0}F(0,-)\frac{\omega_0^n}{n!}&\leq \int_{X_0}\liminf\limits_{k\rightarrow +\infty}F(y_k,-)\frac{\omega_0^n}{n!}\\
&\leq \liminf\limits_{k\rightarrow +\infty}\int_{X_0}F(y_k,-)\frac{\omega_0^n}{n!}.
\end{align*}
This completes the proof of this proposition.
\end{proof}

The lower semicontinuity of the Finsler metric $H$ on $E_k$ does not implies automatically
the upper semicontinuity of its dual metric $H^*$ on $E_k^*$.
But in our case, $H^*$ is indeed upper semicontinuous, as shown in the following proposition.

\begin{prop}\label{prop:dual Hodge metric upper semi-cont.:compact}
With the same notations and assumptions as in Proposition \ref{prop:Hodge metric lower semi-continuous:compact},
For every $\xi\in H^0(Y, E^*_k)$, the function $|\xi|(y):=H^*(\xi(y)):Y\ra [0,+\infty]$ is  upper semi-continuous.
\end{prop}
\begin{proof}
Our proof here is based on the idea in \cite{HPS16} and Theorem \ref{thm:Lm-extension projective}.
Without loss of generality, we assume that $Y=B$, the unit ball in $\mathbb{C}^m$.
It suffices to prove $|\xi|_k$ is upper semi-continuous at the origin of $B$.
We need to show that
\begin{align*}
\limsup\limits_{j\rightarrow +\infty}|\xi|(y_j)\leq |\xi|(0)
\end{align*}
for every sequence $y_1,y_2,\cdots\in B$ which converges to the origin.
We may assume that $|\xi|(y_j)\neq 0$ for all $k\in \mathbb{N}$,
and that the sequence $|\xi|(y_j)$ actually has a  limit.
By the lower semicontinuity of $H$ on $E_k$ as shown in Proposition \ref{prop:Hodge metric lower semi-continuous:compact},
$|\xi|(y_j)\neq +\infty$ for all $j$.
From the definition of the dual metric, for each $j\in \mathbb{N}$,
there is a holomorphic section $u_j\in E_{k,y_j}$, such that $\|u_j\|_k=1$
and  $|\langle \xi(y_j),u_j\rangle|=|\xi|(0)$.
By Theorem \ref{thm:Lm-extension projective}, there are section $s_j\in H^0(X, kK_X\otimes L)$
such that $\int_X|s_j|^{2/k}h^{1/k}\leq C$ for some constant $C>0$ independent of $j$.
By Montel's theorem, we may assume $s_j$ converges uniformly on compact sets of $X$ to some $s\in H^0(X, kK_X\otimes L)$.
Then $\lim_{j\ra\infty}<\xi(0),u_j>=<\xi(0), u:=s(0)>$.
It suffices to prove that $\|u\|_k\leq 1$.
From the proof of Proposition \ref{prop:Hodge metric lower semi-continuous:compact},
each $s_j$ determines a lower semi-continuous function $F_j:B\times X_0\rightarrow [0,+\infty]$ with
\begin{align*}
1=\|u_j\|_k=\left(\int_{X_0}F_j(y_j,-)\frac{\omega_0^n}{n!}\right)^{k/2}.
\end{align*}
In the same way, $s$ determines a lower semi-continuous function $F:B\times X_0\rightarrow [0,+\infty]$.
Since the local weight $e^{-\varphi}$ of $h$ is lower semi-continuous,
and $s_j$ converges uniformly on compact subsets to $s$, we get
\begin{align*}
F(0,-)\leq \liminf\limits_{j\rightarrow +\infty}F_j(y_j,-).
\end{align*}
Then by Fatou's lemma, we complete the proof of this proposition.
\end{proof}

A direct consequence of Proposition \ref{prop:dual Hodge metric upper semi-cont.:compact} is the following

\begin{cor}\label{cor:m-Berg continuous:compact case}
For any $m\geq 1$, the relative $m$-Bergman kernel metric (see \S \ref{subsec:Bergman relative case} for definition) on $mK_{X/Y}\otimes L$ is lower semi-continuous,
namely, the norm of any local holomorphic section of $mK_{X/Y}\otimes L$ with respect to the relative $m$-Bergman kernel metric
is lower semi-continuous.
\end{cor}

\subsection{For families of pseudoconvex domains}
Let $\Omega\subset \mathbb{C}^{m+n}=\mathbb{C}_t^m\times \mathbb{C}_z^n$ be a pseudo-convex domain.
Let $p:\Omega\rightarrow \mathbb{C}^m$ be the natural  projection.
We denote $p(\Omega)$ by $D$ and denote $p^{-1}(t)$ by $\Omega_t$ for $t\in D$.
Let $\varphi$ be a plurisubharmonic function on $\Omega$ and let $k\geq 1$ be an fixed integer.
For an open subset $U$ of $D$, we denote by $\mathcal F(U)$ the space of holomorphic functions $F$ on $p^{-1}(U)$ such that
$\int_{p^{-1}(K)}|F|^{2/k}e^{-\varphi}\leq\infty$ for all compact subset $K$ of $D$.
For $t\in D$, let
$$E_{k,t}=\{F|_{\Omega_t}:F\in\mathcal F(U),\ U\subset D\ \text{open\ and}\ t\in U\}.$$
$E_{k,t}$ is a vector space and we define a norm on it as follows:
$$H(f):=|f|_k=\left(\int_{D_t}|f|^{2/k}e^{-\varphi_t}\right)^{k/2}\leq\infty,$$
where $\varphi_t=\varphi|_{D_t}$.
Let $E_k=\coprod_{t\in D}E_{k,t}$ be the disjoint union of all $E_{k,t}$.
Then we have a natural projection $\pi:E_k\ra D$ which maps elements in $E_{k,t}$ to $t$.
We view $H$ as a Finsler metric on $E_k$.

In general $E_k$ is not a genuine holomorphic vector bundle over $D$.
However, we can also talk about its holomorphic sections,
which are the objects we are really interested in.
By definition, a section $s:D\ra E_k$ is a \emph{holomorphic section} if it varies holomorphically with $t$,
namely, the function $s(t,z):\Omega\ra \mc$ is holomorphic with respect to the variable $t$.
Note that $s(t,z)$ is automatically holomorphic on $z$ for $t$ fixed,
by Hartogs theorem, $s(t,z)$ is holomorphic jointly on $t$ and $z$ and hence is a holomorphic function on $\Omega$.
In some sense, $E_k$ can be viewed as an object similar to holomorphic vector fields studied in \cite{LS14}.



Let $E^*_{k,t}$ be the dual space of $E_{k,t}$,
namely the space of all complex linear functions on $E_{k,t}$.
Let $E^*_k=\coprod_{t\in D}E^*_{k,t}$.
The natural projection from $E^*_k$ to $D$ is denoted by $\pi^*$.
Note that we do not define any topology on $E^*_{k,t}$ and $E_k$.
The only object we are interested in is holomorphic sections of $E^*_k$ which we are going to define.
Given a holomorphic section $s$ of $E_k$ on some open set $U$ of $D$,
$s$ induces a function $|s|_k:U\ra \mr$ with $|s|_k(t)$ given by $|s(t)|_k$,
which is lower semicontinuous and hence measurable, by the following Proposition \ref{prop:lower semi-continuous: noncompact}.

\begin{defn}\label{def:holo section of dual bundle}
A section $\xi$ of $E^*_{k}$ on $D$ is holomorphic if:
\begin{itemize}
\item[(1)] for any local holomorphic section $s$ of $E_k$, $<\xi, s>$ is a holomorphic function;
\item[(2)] for any sequence $s_j$ of holomorphic sections of $E_k$ on $D$ such that
$\int_D|s_j|_k\leq 1$, if $s_j(t,z)$ converges uniformly on compact subsets of $\Omega$ to $s(t,z)$ for some holomorphic section $s$ of $E_k$,
then $<\xi,s_j>$ converges uniformly to $<\xi,s>$ on compact subsets of $D$.
\end{itemize}
\end{defn}

In the same way we can define holomorphic section of $E^*_k$ on open subsets of $D$.
The Finsler metric $H$ on $E_k$ induces a Finsler metric $H^*$ on $E^*_k$,
as defined the Definition \ref{def:dual finsler metric}.
We will show that $H$ is lower semicontinuous and $H^*$ is upper semicontinuous,
as analogues of Proposition \ref{prop:Hodge metric lower semi-continuous:compact} and Proposition \ref{prop:dual Hodge metric upper semi-cont.:compact}
in the case of families of pseudoconvex domains.

\begin{prop}\label{prop:lower semi-continuous: noncompact}
With the above notations and assumptions.
Assume $s$ is a holomorphic section of $E_k$,
then the function $|s|_k(t):=H(s(t)):D\ra [0,+\infty]$ is lower semicontinuous.
\end{prop}
\begin{proof}
We assume $0\in D$ and prove that $|s|_k$ is lower semicontinuous for a point $0$.
Let $K_1\Subset K_2\Subset \cdots\Subset K_j\Subset\cdots\Subset \Omega_0$ be an
increasing sequence of compact subsets of $\Omega_0$, such that $\cup_jK_j=\Omega_0$.
Since the set valued function $t\rightarrow \Omega_t$ is lower semi-continuous,
in the sense that if $\Omega_t$ contains a compact set $K$, then $K$ is contained
in all $\Omega_s$ for $s$ sufficiently close to $t$.
Thus for any $j$, there is a small disk $B_j\subset D$ centered at $a$,
such that $B_j\times K_j\Subset \Omega$.
Note that $e^{-\varphi}$ is lower semicontinuous,
hence $\liminf_{t\ra 0}|s|_k(t)\geq \left(\int_{K_j}|s(0,z)|^{2/k}e^{-\varphi_t}\right)^{k/2}$ for all $j$.
Letting $j$ goes to $\infty$, we get $\liminf_{t\ra 0}|s|_k(t)\geq |s|_k(0)$.
\end{proof}

The following lemma shows that $|\xi|_k(t)$ can not take value $+\infty$ anywhere.

\begin{lem}\label{lem:dual norm finite}
Let $\xi$ be a holomorphic section of $E^*_k$, then $|\xi|_k(t)<+\infty$ for all $t\in D$.
\end{lem}
\begin{proof}
We argue by contradiction.
Assume $0\in D$ and $|\xi(0)|_k=+\infty$.
By definition, there is a sequence $\{u_j\}\subset E_{k,0}$ such that $|u_j|_k=1$ and $\lim_{j\ra \infty}<\xi(0),u_j>=+\infty$.
By Theorem \ref{thm:Lm-extension on psc}, there are holomorphic sections $s_j$ of $E_k$ such that $s_j(0)=u_j$ and
$\int_D|s_j|_k\leq C$ for some constant $C$ independent of $j$.
By Montel's theorem there is a subsequence of $\{s_j\}$, may assumed to be $\{s_j\}$ itself,
that converges uniformly on compact subsets of $\Omega$ to some holomorphic section $s$ of $E_k$.
By definition, $<\xi, s_j>$ converges uniformly on compact sets of $D$ to $<\xi,s>$.
In particular, $<\xi(0), u_j>$ converges to $<\xi(0), s(0)>\leq +\infty$,
which is a contradiction.
\end{proof}

\begin{prop}\label{prop: dual hodge u.s.c domain}
Let $\xi:D\ra E^*_k$ be a holomorphic section of $E^*_k$.
Then the function $|\xi|_k(t):=H^*(\xi(t)):D\ra [0,+\infty]$ is upper semicontinuous.
\end{prop}
\begin{proof}
We assume $0\in D$ and prove that $|\xi|_k$ is upper semicontinuous at $0$.
We need to show that
\begin{align*}
\limsup\limits_{j\rightarrow +\infty}|\xi|_k(t_j)\leq |\xi|_k(0).
\end{align*}
for every sequence $t_1,t_2,\cdots\in D$ which converges to 0.
We may assume that $|\xi|_k(t_j)\neq -\infty$ for all $j\in \mathbb{N}$,
and that the sequence $|\xi|_k(t_j)$ actually has a  limit.
From the definition of the dual metric and Lemma \ref{lem:dual norm finite}, for each $j$, there exists $u_j\in E_{k,t_j}$,
such that $|u_j|_k=1$ and  $|\xi|_k(t_j)<|\langle \xi(t_j),u_j\rangle|+\epsilon$,
where $\epsilon>0$ is an arbitrary constant.
By Theorem \ref{thm:Lm-extension on psc},
there are holomorphic sections $s_j$ of $E_k$ such that
\begin{align*}
s_j(t_j)=u_j~~~~\mbox{~~~~and~~~} \int_D|s_j(t)|_k\leq K
\end{align*}
for some constant $K$ independent of $j$.
By Montel's theorem, there is a subsequence of $\{s_j\}$, may assumed to be $\{s_j\}$ itself,
that converges on compact subsets of $\Omega$ uniformly to some holomorphic section $s$ of $E_k$.
By definition, $<\xi, s_j>$, as holomorphic functions on $D$, converges uniformly on compact subsets of $D$ to $<\xi, s>$.
In particular $\limsup_{j\ra\infty}|\xi|_k(t_j)\leq|<\xi(t_j), u_j>|+\epsilon\ra |<\xi(0), s(0)>|+\epsilon$.
If $s(0)=0$, we are done.
We assume $s(0)\neq 0$.
Then it suffices to prove that $|s(0)|_k\leq 1$.
But this is true since $s_j$ converges to $s$ uniformly on compacts sets,
$|u_j|_k=1$,  and $e^{-\varphi}$ is lower semicontinuous.
\end{proof}

A direct consequence of Proposition \ref{prop: dual hodge u.s.c domain} is the following

\begin{cor}\label{cor:m-Bergman kernel continu}
Let $\Omega, p, \varphi$ as in the beginning of this subsection.
For any positive integer $k$, let $K_k(t,z)$ be the $k$-Bergman kernel (see \S \ref{subsec:m-Bergman kernel} for definition)
on $\Omega_t:=p^{-1}(t)$, $t\in D$, with weight $e^{\varphi_t}$.
Then the relative $k$-bergman kernel $K_k(t,z)$
is upper semi-continuous on $\Omega$.
\end{cor}

\begin{rem}
Let $\xi$ be a holomorphic section of $E^*_k$.
By Lemma \ref{lem:dual norm finite} and Theorem \ref{prop: dual hodge u.s.c domain},
$|\xi|_k(t)$ is locally bounded above by positive constants.
On the other hand, when $k=1$, it is not difficult to show that a section of $E^*_k$
is holomorphic if it satisfies condition (1) in Definition \ref{def:holo section of dual bundle}
and its norm is locally bounded above.
It seems that the same result should be true for general $k$,
but we can not give a proof right off the bat.
\end{rem}

\begin{rem}
For the case that $\Omega=D\times D'$ is a product domain and $\varphi$ is bounded on $\Omega$,
by the mean value inequality, $|f|_k<\infty$ for any $f\in E_{k,t}$, $t\in D$.
If in addition that $k=1$, $E_{1,t}$ consists of square integrable holomorphic functions
on $D'$ with respect to the weight $e^{-\varphi_t}$,
which is the setting considered by Berndtsson in \cite{Bob06}.
\end{rem}

\section{New characterization of plurisubharmonic functions and positively curved vector bundles}\label{sec:char of p.s.h and
posit bdle}

\subsection{Characterization of plurisubharmonic functions}
It is known that a p.s.h function can be used as a weight in the Ohsawa-Takegoshi $L^2$ extension theorem.
In this section, we prove a converse in some sense of this result, namely,
if a function can be used as a weight in the Ohsawa-Takegoshi type $L^p$ extension for some $p>0$,
it is p.s.h. This result is inspired by Demailly's work of regularization of p.s.h functions \cite{Dem92}.

\begin{thm}\label{thm:cha. of p.s.h function}
Let $\varphi:D\ra [-\infty,+\infty)$ be a upper semicontinuous function on $D\subset \mc^n$ that is not identically $-\infty$.
Let $p>0$ is a fixed constant. If for any $z_0\in D$ with $\varphi(z_0)>-\infty$ and any $m>0$, there is $f\in\mathcal O(D)$ such that $f(z_0)=1$ and
$$\int_D|f|^{p}e^{-m\varphi}\leq C_me^{-m\varphi(z_0)},$$
where $C_m$ are constants independent of $z_0$ and satisfying $\log C_m/m\ra 0$,
then $\varphi$ is plurisubharmonic.
\end{thm}

We need some preparation for the proof of Theorem \ref{thm:cha. of p.s.h function}.
Let $D$, $\varphi$, and $p$ as in Theorem \ref{thm:cha. of p.s.h function}.
Let
$$H^p(D,\varphi)=\{f\in\mathcal O(D);|f|_p:=\int_D|f|^pe^{-\varphi}<\infty\}.$$
For $z\in D$, define
$$K_{\varphi,p}(z)=(\inf\{|f|_p; f\in H^p(D,\varphi), f(z)=1\})^{-1}$$
if there exist $f\in H^p(D,\varphi)$ with $f(z)\neq 0$,
and otherwise $K_{\varphi,p}(z)$ is defined to be 0.
It is also easy to see that $K_{\varphi,p}=\sup\{|f(z)|^p; f\in H^p(D,\varphi), |f|_p=1\}$.
We want to show that $K_{\varphi}$ is continuous and $\log K_{\varphi,p}$ is plurisubharmonic.

\begin{lem}\label{lem:Berg continu general p}
With the above notations, $K_{\varphi,p}$ is a continuous function on $D$.
\end{lem}
\begin{proof}
This is proved by an elementary normal family argument.
By definition, it is clear that $K_{\varphi,p}$ is lower semicontinuous.
We now show it is is also upper semicontinuous.
Assume $a\in D$ and $z_j\in D$ which converge to $a$ as $j\ra\infty$.
Let $\epsilon>0$ be arbitrary. There exist $f_j\in H^p(D,\varphi)$ such that $|f_j|_p=1$
and $|f_j(z_j)|^p>K_{\varphi,p}(z_j)+\epsilon$.
Since $\varphi$ is upper semicontinuous, $\{f_j\}$ is a normal family on $D$
and hence have a subsequence, may assumed to be $\{f_j\}$ itself,
that converges uniformly to some $f\in \mathcal O(D)$ on compact sets of $D$.
It is clear that $f\in H^p(D,\varphi)$ and $|f|_p\leq 1$.
So $$K_{\varphi,p}(a)\geq |f(a)|^p=\lim_{j\ra\infty}|f_j(z_j)|^p\geq \limsup_{j\ra\infty}K_{\varphi,p}(z_j)-\epsilon.$$
Letting $\epsilon$ goes to 0, we see $K_{\varphi,p}$ is upper semicontinuous.
\end{proof}

\begin{lem}\label{lem:Berg p.s.h general p}
$\log K_{\varphi,p}$ is a plurisubharmonic function on $D$.
\end{lem}
\begin{proof}
Note that $\log K_{\varphi,p}=\sup\{p\log|f|; f\in H^p(D,\varphi), |f|_p=1\}$
and $\log K_{\varphi,p}$ is upper semicontinuous by Lemma \ref{lem:Berg continu general p},
$\log K_{\varphi,p}$ is  plurisubharmonic.
\end{proof}

We now give the proof of Theorem \ref{thm:cha. of p.s.h function}:
\begin{proof}
We will use the above notations and definitions.
We denote $\frac{1}{m}\log K_{m\varphi,p}(z)$ by $\varphi_m(z)$.
By Lemma \ref{lem:Berg p.s.h general p}, $\varphi_m$ is p.s.h on $D$.
We want to show that $\varphi_m$ converges to $\varphi$ as $m\ra\infty$.

By assumption, we know
$$-\varphi_m(z)\leq \frac{1}{m}\log\left(e^{-m\varphi(z)}\right)+\frac{\log C_m}{m}=-\varphi(z)+\frac{\log C_m}{m}.$$
This is
$$\varphi_m(z)\geq \varphi(z)-\frac{\log C_m}{m}.$$

On the other hand, if $d(z,\partial D)>r$, the mean value inequality for p.s.h functions implies
$$\int_D|f(\zeta)|^{2p}e^{-m\varphi(\zeta)}\geq \frac{\pi r^n}{n!}e^{-m\sup_{\zeta\in B(z,r)}\varphi(\zeta)},$$
where $f$ is the minimal solution in the definition of $K_{m\varphi,p}(z)$ and $B(z,r)=\{\zeta\in D;|\zeta-z|\leq r\}$.
So we get
$$\varphi_m(z)\leq \sup_{\zeta\in B(z,r)}\varphi(\zeta)-\frac{1}{m}\log(\frac{\pi^n r^n}{n!}).$$

So get
$$\varphi(z)-\frac{\log C_m}{m}\leq\varphi_m(z)\leq \sup_{\zeta\in B(z,r)}\varphi(\zeta)-\frac{1}{m}\log(\frac{\pi^n r^n}{n!}).$$
We now take $r=e^{-\sqrt{m}/n}$, then the above inequality becomes
$$\varphi(z)-\frac{\log C_m}{m}\leq\varphi_m(z)\leq \sup_{\zeta\in B(z,e^{-\sqrt{m}/n})}\varphi(\zeta)-\frac{1}{m}\log\frac{\pi^n}{n!}+\frac{1}{\sqrt{m}}.$$
Note that $\varphi$ is u.s.c, the above inequality implies that $\limsup_{\zeta\ra z}\varphi(\zeta)=\varphi(z)$
and $\varphi_m$ converges to $\varphi$ pointwise as $m\ra\infty$.
Let $\psi_m=\sup_{j\geq m}\varphi_m$ and let $\psi_m^*$ be the upper semicontinuous regularization of $\psi_m$.
We have
$$\varphi(z)\leq \psi_m(z)\leq \sup_{\zeta\in B(z,e^{-\sqrt{m}/n})}\varphi(\zeta)+\frac{2}{\sqrt{m}}$$
for $m>>1$. Since the right hand side term of the above inequality is u.s.c, $\psi^*_m$ also satisfies the same inequality.
So we also have $\psi^*_m$ converges to $\varphi$ pointwise as $m\ra\infty$, thus $\varphi$ is plurisubharmonic.

\end{proof}


\subsection{Characterization of positive vector bundles}\label{subsec:char. of positive v.b.}
We now generalize Theorem \ref{thm:cha. of p.s.h function} in the previous section
to vector bundles. For simplicity, we only consider holomorphic vector bundles of finite rank here,
but the argument here can be applied to a very general framework,
as we will see in the following sections.

We start from some basic linear algebra.
Let $V$ be a vector space of finite dimension.
Recall that a Finsler metric on $V$ is defined to be a map $h:V\ra [0,+\infty]$ such that $h(cv)=|c|^2h(v)$
for all $v\in V$ and $c\in\mc$.
Given a Finsler metric on $V$, the dual metric $h^*$ on the dual space $V^*$ is defined as in Definition \ref{def:dual finsler metric}.
For a positive integer $m$, the $m$-th tensor power of $V$ is denoted by $V^{\otimes m}$.
Then $h$ and $h^*$ induces naturally Finsler metrics $h^m$ and $h^{*m}$ on $V^{\otimes m}$ and $(V^*)^{\otimes m}$ as follows.
Recall that a vector $\xi\in (V^*)^{\otimes m}$ can be viewed as a map $\xi:V^m\ra\mc$
which is multilinear, namely linear on each components.

\begin{defn}\label{def:metric on tensor power}
The metric $h^{*m}:(V^*)^{\otimes m}\ra [0,+\infty]$ on $(V^*)^{\otimes m}$ is defined as:
$$h^{*m}(\xi):=\sup\{|\xi(u_1,\cdots,u_m)|; u_i\in V, h(u_i)\leq 1, 1\leq i\leq m\}$$
if $h(u)<+\infty$ for some $u\in V$, otherwise $h^{*m}(\xi)$ is defined to be 0.
The metric $h^m$ on $V^{\otimes m}$ is defined in the same way by identifying $V$ and $(V^*)^*$,
the dual space of $V^*$.
\end{defn}

According to this definition, for $\xi_1,\cdots, \xi_m\in V^*$,
we have the product formula $h^{*m}(\xi_1\cdots \xi_m)=h^*(\xi_1)\cdots h^*(\xi_m)$.
Definition \ref{def:metric on tensor power} can be applied to holomorphic vector bundles.
If $E$ is a holomorphic vector bundle over a complex manifold $X$ and $h$ is a Finsler metric on $E$.
The induced metrics $h^m$ and $h^{*m}$ on $E^{\otimes m}$ and $(E^*)^{\otimes m}$ is defined pointwise.

We now introduce a notion that may have independent interest.

\begin{defn}[Multiple $L^p$-extension property]\label{def:multiple extension prop}
Let $(E,h)$ be a holomorphic vector bundle over a bounded domain
$D\subset\mc^n$ equipped with a singular Finsler metric $h$. Let
$p>0$ be a fixed constant. Assume that for any $z\in D$, any nonzero element $a\in E_{z}$
with finite norm $|a|$, and any $m\geq 1$, there is a holomorphic
section $f_m$ of $E^{\otimes m}$ on $D$ such that $f_m(z)=a^{\otimes
m}$ and satisfies the following estimate:
$$\int_D|f_m|^{p}\leq C_m|a^{\otimes{m}}|^{p}=C_m |a|^{mp},$$
where $C_m$ are constants independent of $z$ and satisfying the
growth condition $\frac{1}{m}\log C_m\ra 0$ as $m\ra\infty$. Then
$(E, h)$ is said to have multiple $L^p$-extension property.
\end{defn}

The following theorem says that multiple $L^p$-extension property for some $p>0$ implies Griffiths positivity.

\begin{thm}\label{thm:cha of positive bundle}
Let $(E,h)$ be a holomorphic vector bundle over a bounded domain
$D\subset\mc^n$ equipped with a singular Finsler metric $h$, such
that the norm of any local holomorphic section of $E^*$ is upper
semicontinuous. If $(E,h)$ has multiple $L^p$-extension property for
some $p>0$, then $(E,h)$ is positively curved in the sense of
Griffiths, namely $\log|u|$ is plurisubharmonic for any local
holomorphic section $u$ of $E^*$.
\end{thm}
\begin{proof}
Let $u$ be a local holomorphic section of $E^*$ on $U\subset D$.
We need to show that the function $\varphi:=\log|u|$ is plurisubharmonic on $U$.
Our strategy is to prove that $\varphi$ satisfies the condition in Theorem \ref{thm:cha. of p.s.h function}.

Without loss of generality, we assume $U=D$.
Let $z$ be a fixed point in $D$.
We assume that $|u(z)|\neq 0$.
Let $a\in E_z$ such that $|a|=1$ and $<u(z),a>=|u(z)|$.
By assumption, there is a holomorphic section $f$ of $E^{\otimes m}$ over $D$ such that
$f(z)=a^{\otimes m}$ and $\int_D|f|^{p}\leq C_m|a^{\otimes{m}}|^{p}=C_m|a|^{mp}=C_m$.
We view $u^{\otimes m}$ as a holomorphic section of $(E^*)^{\otimes m}$.
It is obvious that $|u^{\otimes m}|=|u|^m$ and $|u^{\otimes m}(z)|=<u^{\otimes m}(z),a^{\otimes m}>$.
By definition,
$$|u(\zeta)|^m\geq |<u^{\otimes m}(\zeta), f(\zeta)>|/|f(\zeta)|$$
for $\zeta\in D,$
which is
\be\label{eqn:weight estimate}
e^{-m\varphi(\zeta)}\leq e^{-\log|<u^{\otimes m}(\zeta),f(\zeta)>|}|f(\zeta)|.
\ee
Since $u^{\otimes m}, f$ are holomorphic section of $(E^*)^{\otimes m}$ and $E^{\otimes m}$ respectively,
$<u^{\otimes m}, f>$ is a holomorphic function on $D$.
By the Ohsawa-Takegoshi extension theorem, there is a holomorphic function $h$ on $D$ such that
$h(z)=1$ and
$$\int_D|h|^2e^{-p\log|<u^{\otimes m}(\zeta),f(\zeta)>|}\leq Ce^{-p\log|<u^{\otimes m}(z),f(z)>|}=Ce^{-pm\varphi(z)},$$
where $C$ is a constant independent $m$ and $z$.
By the above inequality, we have
\begin{equation*}
\begin{split}
\int_D|h|e^{-\frac{p}{2}m\varphi}
& \leq \int_D|h|e^{-\frac{p}{2}\log|<u^{\otimes m}(\zeta),f(\zeta)>|}|f|^{\frac{p}{2}}\\
& \leq \left(\int_D|h|^2e^{-p\log|<u^{\otimes m}(\zeta),f(\zeta)>|}\int_D|f|^{p}\right)^{1/2}\\
& \leq \left(Ce^{-pm\varphi(z)}C_m \right)^{1/2}\\
& = \sqrt{CC_m}e^{-\frac{p}{2}m\varphi(z)}.
\end{split}
\end{equation*}

By Theorem \ref{thm:cha. of p.s.h function}, $\varphi=\log|u|$ is p.s.h on $D$.
\end{proof}

\section{Proof of Berndtsson's minimum principle}\label{sec:minimum principle}
In this section we give a proof of Berndtsson's integral form of minimum principle.
The proof is motivated by Berndtsson's original proof, but our starting point is Theorem \ref{thm:cha. of p.s.h function}
and the main ingredient is Ohsawa-Takegoshi extension theorem,
while the main ingredient in Berndtsson's original proof is H\"ormander's $L^2$ estimate of $\bar\partial$.
Of course Theorem \ref{thm:Berndtsson MP} is a corollary of the plurisubharmonic variation of relative Bergman kernels
that will be discussed in \S \ref{sec:p.s.h var of m-bergman},
but we also present it here to show in this simple case how the main ideas work.

\begin{thm}[\cite{Bob98}]\label{thm:Berndtsson MP}
    Let $\Omega\subset\mc^r_z\times\mc^n_w$ be a pseudoconvex domain
    and let $p:\Omega\ra U:=p(\Omega)\subset\mc^r$ be the natural projection.
    Let $\varphi(z,w)$ be a plurisubharmonic function on $\Omega$.
    If all fibers $\Omega_z:=p^{-1}(z)$ $(z\in p(\Omega))$ are Reinhardt domains in $\mc^n$
    and $\varphi(z,e^{i\theta}w)=\varphi(z,w)$ for all $\theta\in\mr^n$.
    Then the function $\tilde\varphi$ defined by
    $$e^{-\tilde\varphi(z)}=\int_{\Omega_z}e^{-\varphi(z,w)}d\lambda(w)$$
    is a plurisubharmonic function on $p(\Omega)$,where $d\lambda(w)$ is the Lebesgue measure on $\mc^n$.
\end{thm}
\begin{proof}
It is clear that $\tilde\varphi$ is upper semicontinuous,
so it suffices to prove that $\varphi$ satisfies the condition in Theorem \ref{thm:cha. of p.s.h function}.
For $m\geq 1$, let
$$\Omega_m=\{(u_1,\cdots, u_m)\in\Omega^m; p(u_1)=\cdots=p(u_m)\}$$
be the fiber product of $\Omega$ and let $D^m_z=\{z\}\times \Omega_z^m\subset \Omega_m$ for $z\in U$.
We can naturally identify $\Omega_m$ with the disjoint union $\coprod_{z\in U}D^m_z$.
For any $z\in U$ with $\int_{\Omega_z}e^{-\varphi(z,w)}d\lambda(w)<\infty$,
by Theorem \ref{thm:Demailly OT}, there is $f\in \mathcal O(\Omega_m)$ such that  $f|_{D^m_z}\equiv 1$ and
$$\int_{\Omega_m}|f(\zeta,w_1,\cdots, w_m)|^2e^{-\sum_{i=1}^m\varphi(z,w_i)}\leq C\int_{D^m_z}e^{-\sum_{i=1}^m\varphi(z,w_i)}=e^{-m\tilde\varphi(z)},$$
where $C$ is a constant independent of $z$ and $m$.
We choose $f$ such that the left hand side in the above inequality is minimal.
By the uniqueness of the minimal element, $f(\zeta, e^{i\theta_1}w_1, \cdots, e^{i\theta_m}w_m)=f(\zeta,w_1,\cdots, w_m)$ for all $\theta_i\in\mr^n$.
So $f$ is independent of $w_1, \cdots, w_m$.
We denote $f(\zeta, w_1,\cdots, w_m)$ by $g(\zeta)$.
From the above inequality, we have $g(z)=1$ and
$$\int_U|g(\zeta)|^2e^{-m\tilde\varphi(\zeta)}\leq Ce^{-m\tilde\varphi(z)}.$$
By Theorem \ref{thm:cha. of p.s.h function}, $\tilde\varphi$ is subharmonic on $U$.
\end{proof}

\begin{rem}
Theorem \ref{thm:Berndtsson MP} can be generalzied to holomorphically convex K\"ahler manifolds
with more general group actions, which is the framework for minimum principle considered in \cite{DZZ14, DZZ17}.
\end{rem}

\section{Plurisubharmonic variation of relative $m$-Bergman kernel metrics}\label{sec:p.s.h var of m-bergman}
The aim of this section is to prove that the relative $m$-Bergman kernel metric associated
to a family of pseudoconvex domains or compact K\"ahler manifolds have semipositive
curvature current.

\subsection{For families of pseudoconvex domains}\label{subsec:Bergman kernel p.s.h domain case}
Let $\Omega\subset\mc^r_t\times\mc^n_z$ be a pseudoconvex domain
and let $p:\Omega\ra U:=p(\Omega)\subset\mc^r$ be the natural projection.
Let $\varphi(t,z)$ be a plurisubharmonic function on $\Omega$.
Let $\Omega_t:=p^{-1}(t)$ $(t\in U)$ be the fibers.
We also denote by $\varphi_t(z)=\varphi(t,z)$ the restriction of $\varphi$ on $\Omega_t$.
Let $m$ be a positive integer, and let $K_{m,t}(z)$ be the $m$-Bergman kernel on $\Omega_t$ with respect to the weight $e^{-\varphi_t}$.
As $t$ varies, we will view $K_{m,t}(z)$ as a function on $\Omega$, which will be called the relative $m$-Bergman kernel on $\Omega$.
When $m=1$, $K_{1,t}(z)$ is the ordinary relative Bergman kernel.
Berndtsson proved that $K_{1,t}(z)$ is log-plurisubharmonic on $U$,
based on regularity of $\bar\partial$-Neumann and H\"{o}rmander's $L^2$-estimate to $\bar\partial$.
In this section, we prove that the relative $m$-Bergman kernel is log-plurisubharmonic for general $m$.
Our main techniques are Theorem \ref{thm:cha. of p.s.h function} and raising powers of the domains.

\begin{thm}\label{thm:psh variation of m Bergman kernel}
With the above assumptions and notations, for any $m\geq 1$
the function $\psi_m(t,z):=\log K_{m,t}(z)$ is a plurisubharmonic function on $\Omega$.
\end{thm}
\begin{proof}
We may assume that $U$ is bounded.
By Corollary \ref{cor:m-Bergman kernel continu},
$\psi_m(t,z)$ is upper semicontinuous.
It suffices  to prove that $\psi_m(t, s(t))$ is subharmonic for all holomorphic sections $s:U\rightarrow \Omega$.
We want to prove that $\psi_m(t, s(t))$ satisfies the condition in Theorem \ref{thm:cha. of p.s.h function} on $U$.

For $k\geq 1$, let
$$\Omega_k=\{(u_1,\cdots, u_k)\in\Omega^k; p(u_1)=\cdots=p(u_k)\}$$
be the fiber product of $\Omega$ and let $D^k_t=\{t\}\times \Omega_t^k\subset \Omega_k$ for $t\in U$.
We can naturally identify $\Omega_k$ with the disjoint union $\coprod_{t\in U}D^k_t$.
Let $\varphi_k(t,z_1,\cdots, z_k):=\varphi(t,z_1)+\cdots+\varphi(t,z_k)$,
then it is a plurisubharmonic function on $\Omega_k$.

Let $t_0\in U$ and $s:U\rightarrow \Omega$ be a holomorphic section
such that $\psi_m(t_0,s(t_0))\neq -\infty$.
By definition of the $m$-Bergman kernel and Proposition \ref{prop:product property},
there is $f\in \mathcal{O}(D_{t_0}^m)$ with $f(t_0,s(t_0),\cdots, s(t_0))=1$ and
\begin{align*}
\int_{D^k_{t_0}}|f|^{2/m}e^{-\varphi_k(t_0,z_1,\cdots,z_k)/m}=e^{-k\psi_m(t_0,s(t_0))/m}.
\end{align*}
By Theorem \ref{thm:Lm-extension on psc},  there is an $F\in \mathcal{O}(\Omega_k)$ such that $F|_{D_{t_0}^k}=f$ and
\begin{align*}
\int_{\Omega_k}|F|^{2/m}e^{-\varphi_k(t,z_1,\cdots,z_k)/m}\leq C\int_{D^k_{t_0}}|f|^{2/m}e^{-\varphi_k(t_0,z_1,\cdots,z_k)/m},
\end{align*}
where $C$ is a constant independent of $t_0, s,$ $k$ and $m$.   The left hand side of the above inequality is
$$\int_U\int_{D^k_t}|F(t,z_1,\cdots,z_k)|^{2/m}e^{-\varphi_k(t,z_1,\cdots,z_k)/m}.$$
By definition of the $m$-Bergman kernel and Proposition \ref{prop:product property},  it is larger than
$$\int_U|F(t,s(t),\cdots,s(t))|^{2/m}e^{-k\psi_m(t,s(t))/m}.$$
Let $g(t)=F(t,s(t),\cdots,s(t))$, then $g$ is a holomorphic function on $U$ and $g(t_0)=1$, and satisfies the following estimate
$$\int_{U}|g|^{2/m}e^{-k\psi_m(t,s(t))/m}\leq Ce^{-k\psi_m(t_0,s(t_0))/m},$$
where $C$ is a constant independent of $m$.
By Theorem \ref{thm:cha. of p.s.h function}, $\psi_m(t,s(t))$ is a plurisubharmonic function on $U$.
\end{proof}

\subsection{For families of compact K\"ahler manifolds}\label{subsec:Bergman kernel p.s.h kahler manifolds}
Let $X, Y$ be K\"ahler  manifolds of  dimension $r+n$ and $r$ respectively,
let $p:X\rightarrow Y$  be a proper holomorphic mapping.
Let $L$ be a holomorphic line bundle over $X$, and $h$ be a singular  Hermitian metric on $L$,
whose curvature current is semi-positive. Let $K_{X/Y}$ be the relative canonical bundle on $X$.

Let $\mathcal E_m=p_*(mK_{X/Y}\otimes L\otimes \mathcal I_m(h))$ be the direct image sheaf on $Y$.
By Grauert's theorem, $\mathcal E_m$ is a coherent analytic sheaf on $Y$.
There is an Zariski open subset $U$ in $Y$ such that $\mathcal E_m|_U$ is locally free
and the fiber $E_{m,y}$ over $y\in U$ of the vector bundle associated to $\mathcal E_m|_U$, denoted by $E_m$,
can be identified with $H^0(X_y, (kK_{X/Y}\otimes L\otimes \mathcal I_k(h))|_{X_y})\subset H^0(X_y, kK_{X_y}\otimes L|_{X_y})$.

Recall that we have define the relative $m$ Bergman kernel metric, denoted by $B^m_{L,h}$, on $(mK_{X/Y}\otimes L)|_{p^{-1}(U)}$
in \S \ref{subsec:Bergman relative case}.
The main result in this subsection is the following

\begin{thm}\label{thm:p.s.h variation Berg compact}
If $h$ is locally bounded, the relative $m$ Bergman kernel metric $B^m_{L,h}$ on $(mK_{X/Y}\otimes L)|_{p^{-1}(U)}$ has nonnegative curvature current;
moreover, $B^m_{L,h}$ can be extended to a hermitian metric on $mK_{X/Y}\otimes L$ whose curvature current is nonnegative.
\end{thm}
\begin{proof}
The proof of the first statement is similar to the proof of Theorem \ref{thm:psh variation of m Bergman kernel}.
We may assume that $U=\mathbb B^r$ is the unit ball in $\mc^r$.
By Corollary \ref{cor:m-Berg continuous:compact case}, $B^m_{L,h}$ is lower semi-continuous.

For simplicity, we denote the line bundle $mK_{X/Y}\otimes L$ on $X$ by $F$ and denote $p^{-1}(U)$ by $\Omega$.
Let $s:U\rightarrow \Omega$ be an arbitrary holomorphic map.
Let $W$ be a neighborhood of the image of $h$ in $\Omega$ such that $F$
has a holomorphic frame $e$ on $W$.
Let $\psi$ be the weight of the metric $B^m_{L,h}$ on $W$ with respect to the frame $e$.
It suffices to prove that $\psi\circ s$ is a plurisubharmonic function on $U$ (or identically equal to $-\infty$).
We will prove that $\psi\circ s$ satisfies the condition in Theorem \ref{thm:cha. of p.s.h function}.

For $k\geq 1$, let
$$\Omega_k=\{(x_1,\cdots, x_k)\in \Omega^k; p(x_1)=\cdots=p(x_k)\}$$
be the $k$-th fiber product power of $\Omega$ with respect to the map $p:\Omega\rightarrow U$.
Let $p_k:\Omega_k\rightarrow U$ be the natural projection,
and let $\Omega^k_{t}=p^{-1}_k(t)$ for $t\in U$.
The line bundle $L$ on $\Omega$ induces a line bundle $L_k$ on $\Omega_k$ whose
fiber at $(x_1,\cdots, x_k)$ is $L_{x_1}\otimes\cdots\otimes L_{x_k}$.
The metric $h$ on $L$ naturally induces a metric $h_k$ on $L_k$ whose curvature current is nonnegative.

Let $t_0\in U$ be an arbitrary point such that $\psi(s(t_0))\neq -\infty$.
We denote $s(t_0)\in \Omega$ by $z$  and denote $e(s(t_0))$ by $a$.
Note that $(z,\cdots, z)\in \Omega_k$ and we can identify
$(mK_{\Omega_k/U}\otimes L_m)|_{(z,\cdots,z)}$ with $(mK_{X/Y}\otimes L)|^{\otimes k}_{z}=F_{z}\otimes\cdots\otimes F_{z}$.
By Proposition \ref{prop:product property-manifold}, $a^{\otimes k}:=a\otimes\cdots\otimes a\in (mK_{\Omega_k/U}\otimes L_k)|_{(z,\cdots,z)}$
has square norm $e^{-k\psi(z)}$  with respect to the relative $m$-Bergman kernel metric on $mK_{\Omega_k/U}\otimes L_k$.

By definition of the $m$-Bergman kernel metric,
there is $f\in H^0(\Omega^k_{t_0},mK_{\Omega^k_{t_0}}\otimes L_m|_{\Omega^k_{t_0}})$
such that $f(z,\cdots, z)=a^{\otimes k}$, and
$$\int_{\Omega^k_{t_0}}|f|^{2/m}h_k^{1/m}=e^{-k\psi(z)/m}.$$
By Theorem \ref{thm:Lm-extension projective},  there is an $F\in H^0(\Omega_k, mK_{\Omega_k}\otimes L_k)$ such that $F|_{\Omega_{t_0}^k}=f\wedge dt_1\wedge\cdots\wedge dt_r$ and
\begin{align}\label{eqn:exis of extension if bounded}
\int_{\Omega_k}|F|^{2/m}h^{1/m}_k\leq C\int_{\Omega^k_{t_0}}|f|^{2/m}h^{1/m}_k,
\end{align}
where $(t_1,\cdots, t_r)$ are the coordinates on $U$ and $C$ is a constant independent of $t_0$ and $k$.
The left hand side of the above inequality is
$$\int_U\int_{\Omega^k_t}|(F|_{\Omega^k_t})|^{2/m}h^{1/m}_k.$$
Locally we write $F(x_1,\cdots, x_k)=\tilde F(x_1,\cdots, x_k)e^{\otimes k}$.
By definition of the $m$-Bergman kernel and Proposition \ref{prop:product property-manifold},  it is larger than
$$\int_U|\tilde F(t,s(t),\cdots,s(t))|^{2/m}e^{-k\psi(t,s(t))/m}.$$
Let $g(t)=\tilde F(t,s(t),\cdots,s(t))$, then $g$ is a holomorphic function on $U$ and $g(t_0)=1$, and satisfies the following estimate
$$\int_{U}|g|^{2/m}e^{-k\psi(t,s(t))/m}\leq Ce^{-k\psi(t_0,s(t_0))/m},$$
where $C$ is a constant independent of $k$.
By Theorem \ref{thm:cha. of p.s.h function}, $\psi(t,s(t))$ is a plurisubharmonic function on $U$.

For the proof of the second statement,
it suffices to prove that the Bergman kernel metric is bounded locally below by positive constants.
We follow the idea in \cite{BP08}.  Fix an arbitrary point $p\in X_y$.  Let $u\in H^0(X_y, mK_{X_y}\otimes L|_{X_y})$ such that
\begin{align*}
\int_{X_y}|u|^{2/m}h|_{X_y}^{1/m}\leq 1.
\end{align*}
Let $\Omega\subset X $ be a bounded pseudoconvex  neighborhood of $p$ in $X$. From Theorem \ref{thm:Lm-extension on psc}, there is a local $m$-canonical form $\widetilde{U}$ on $\Omega$  which extends $u$
 and satisfies
 \begin{align*}
 \int_{\Omega}|\widetilde{U}|^{2/m}h^{1/m}\leq C_0\int_{\Omega_y}|u|^{2/m}h|_{\Omega_y}^{1/m}\leq C_0,
 \end{align*}
where $C_0$ is an absolute constant. We write as $\widetilde{U}=U' (dz)^{\otimes m}$. By the mean value inequality on a polydisk  $D_{r_0}$ of poly-radius $(r_0,\cdots, r_0)$ centered at $p$, we  obtain that
\begin{align*}
|u(p)|^{2/m}=|U'(p)|^{2/m}&\leq \frac{1}{(\pi r_0^2)^{r+n}}\int_{D_{r_0}}|U'|^{2/m}d\lambda_z\\
&\leq  \frac{1}{(\pi r_0^2)^{r+n}}\sup_{D_{r_0}}h^{-1/m}\int_{D_{r_0}}|\widetilde{U}|^{2/m}h^{1/m}\\
&\leq C,
\end{align*}
where $C$ is a constant which does not depend on the geometry of the fiber $X_y$, but on the ambient manifold $X$. We thus complete the proof of the theorem.
\end{proof}

\begin{rem}\label{rem:about boundedness of wt}
The boundedness assumption about $h$ in Theorem \ref{thm:p.s.h variation Berg compact}
is to ensure that the extension $F$ exists and satisfies the estimate in \eqref{eqn:exis of extension if bounded}.
When $m=1$, this assumption is not necessary, since by Theorem \ref{thm:Demailly OT},
the extension with estimate always exists.
On the other hand, the result in Theorem \ref{thm:p.s.h variation Berg compact} still holds
without the assumption that $h$ is locally bounded,
because applying similar argument as in \cite{BP08} one can show that Theorem \ref{thm:p.s.h variation Berg compact}
is indeed a consequence of Theorem \ref{thm:Positivity of hodge compact case} that we will prove in \S \ref{sec:positivity of hodge}.
But we will not present the details of the argument of Berndtsson-P\u{a}un in the present article.
\end{rem}

\section{Positivity of direct images of twisted relative canonical bundles}\label{sec:positivity of hodge}


In this section, we generalize the idea in \S \ref{sec:p.s.h var of m-bergman} to show that the direct image of relative canonical bundles twisted by pseudoeffective line bundles
associated to certain families of pseudoconvex domains or compact K\"ahler manifolds are semi-positive in the sense of Griffiths.

\subsection{For families of pseudoconvex domains}
Let $U, D$ be bounded pseudoconvex domains in $\mc^r$ and $\mc^n$ respectively, and let $\Omega=U\times D\subset\mc^r\times\mc^n$.
Let $\varphi$ be a p.s.h function on $\Omega$,
which is for simplicity assumed to be bounded.
For $t\in U$, let $D_t=\{t\}\times D$ and $\varphi_t(z)=\varphi(t,z)$.
Let $E_t=H^2(D_t, e^{-\varphi_t})$ be the space of $L^2$ holomorphic functions on $D_t$ with respect to the weight $e^{-\varphi_t}$.
Then $E_t$ are Hilbert spaces with the natural inner product.
Since $\varphi$ is assumed to be bounded on $\Omega$, all $E_t$ for $t\in U$ are equal as vector spaces,
however, the inner products on them depend on $t$ if $\varphi(t,z)$ is not constant with $t$.
So, under the natural projection,  $E=\coprod_{t\in\Delta}E_t$ is a trivial holomorphic vector bundle (of infinite rank) over $U$
with varying Hermitian metric.

In \cite{Bob06} , Berndtsson proved that $E$ is semipositive in the sense of Griffths, namely,
for any local holomorphic section $\xi$ of the dual bundle $E^*$ of $E$,
the function $\log|\xi|$ is plurisubharmonic (indeed Berndtsson proved a stronger
result which says that $E$ is semipositive in the sense of Nakano).
The aim here is to provide a new proof of the positivity of $E$,
based on the new charaterization of plurisubharmonic functions (Theorem \ref{thm:cha. of p.s.h function})
and the technique of rising powers of domains.

\begin{thm}\label{thm:Positivity of hodge domain case}
The vector bundle $E$ is semipositive in the sense of Griffths.
\end{thm}

Before giving the proof of Theorem \ref{thm:Positivity of hodge domain case},
we first recall the notion of \emph{Hilbert tensor product} of Hilbert spaces and prove a related lemma.
Let $V$ and $W$ be two Hilbert spaces.
For $v\in V, w\in W$, the norm of $v\otimes w$ is defined to be $\|v\|\|w\|$.
If $\{v_i\}_{i\in I}$ and $\{w_j\}_{j\in J}$ are  orthonormal basis' of $V$ and $W$ respectively,
then the Hilbert tensor product $V\hat\otimes W$ of $V$ and $W$ is defined to be the Hilbert space with $\{v_i\otimes w_j\}_{i\in I, j\in J}$ as an orthonormal basis.
It is easy to show that the definition of $V\hat\otimes W$ is independent of the choice of the orthonormal basis'
of $V$ and $W$.
By definition one can check that $(V\hat\otimes W)^*=V^*\hat\otimes W^*$.
The definition can be naturally generalized to the tensor product of several Hibert spaces.
In particular, we can define the tensor powers $V^{\hat\otimes k}:=V\hat\otimes\cdots\hat\otimes V$ ($(k\geq 1)$) of a Hilbert space $V$.
Let $V$ be a Hilbert space.
For $v\in V$, it is obvious that the norm of $v^{\otimes k}:=v\otimes \cdots\otimes v\in V^{\hat\otimes k}$ is $\|v\|^k$ for all $k\geq 1$.

\begin{lem}\label{lem:product property Bergman space}
Let $D_1, D_2$ be bounded domains in $\mc_z^n$ and $\mc_w^m$ respectively.
Let $\phi_1$ and $\phi_2$ be plurisubharmonic functions on $D_1$ and $D_2$.
Then
$$H^2(D_1\times D_2, e^{-(\varphi_1+\varphi_2)})=H^2(D_1,e^{-\varphi_1})\hat\otimes H^2(D_2,e^{-\varphi_2}).$$
\end{lem}
\begin{proof}
Let $\{f_i\}^{\infty}_{i=1}$ and $\{g_i\}^{\infty}_{i=1}$ be  orthonormal basis' of
$H^2(D_1,e^{-\varphi_1})$ and $H^2(D_2,e^{-\varphi_2})$ respectively.
Then $K_1(z)=\sum_i|f_i(z)|^2$ is the Bergman kernel of $H^2(D_1,e^{-\varphi_1})$
and $K_2(w)=\sum_i|g_i(w)|^2$ is the Bergman kernel of $H^2(D_2,e^{-\varphi_2})$.
Let $K(z,w)=\sum_{i,j}|f_i(z)g_j(w)|^2$. It is clear that $K(z,w)=K_1(z)K_2(w)$.
By Fubini theorem. $\{f_i(z)g_j(w)\}^{\infty}_{i,j=1}$ is an orthonormal set of $H^2(D_1\times D_2, e^{-(\varphi_1+\varphi_2)})$.
By Proposition \ref{prop:product property}, the Bergman kernel of $H^2(D_1\times D_2, e^{-(\varphi_1+\varphi_2)})$ equals to $K_1(z)K_2(w)$.
So $\{f_i(z)g_j(w)\}^{\infty}_{i,j=1}$ is an orthonormal basis of  $H^2(D_1\times D_2, e^{-(\varphi_1+\varphi_2)})$
and hence
$$H^2(D_1\times D_2, e^{-(\varphi_1+\varphi_2)})=H^2(D_1,e^{-\varphi_1})\hat\otimes H^2(D_2,e^{-\varphi_2}).$$
\end{proof}

It is clear that Lemma \ref{lem:product property Bergman space} can be generalized to product of  several domains.
We now give the proof of Theorem \ref{thm:Positivity of hodge domain case}.

\begin{proof}
Let $u$ be a local holomorphic section of the dual bundle $E^*$ of $E$.
We need to prove that $\log|u(t)|$ is a plurisubharmonic function.
Without loss of generality, we assume that $u$ is a global holomorphic section,
namely a holomorphic section of $E^*$ on $U$.
The upper semi-continuity of $\log|u(t)|$ follows from Proposition \ref{prop: dual hodge u.s.c domain}.
We now prove that $\log|u(t)|$ satisfies the condition in Theorem \ref{thm:cha. of p.s.h function} for some $p>0$.

For $m\geq 1$, let $\Omega_m=U\times D^m$ and $\varphi_m(t,z_1,\cdots, z_m)=\varphi(t,z_1)+\cdots+\varphi(t,z_m)$.
For $t\in U$, we denote $t\times D^m$ by $D^m_t$.
Let $E^{\hat\otimes m}_t=H^2(D^m_t, e^{-\varphi_m})$, and $E^{\hat\otimes m}=\coprod_{t\in U}E^{\hat\otimes m}_t$.
By Lemma \ref{lem:product property Bergman space},
$E^{\hat\otimes m}$ is the $m$-th tensor power of $E$ in the Hilbert space sense.

Let $t_0\in U$ be an arbitrary point such that $u(t_0)\neq 0$.
By the definition of tensor powers of Hilbert spaces given as above,
$u^{\otimes m}$ is a nonvanishing holomorphic section of $(E^{\hat\otimes m})^*=(E^*)^{\hat\otimes m}$,
and $|u^{\otimes m}(t)|=|u(t)|^m$. Let $f\in E^{\hat\otimes m}_{t_0}$ such that
$$\int_{D^m_{t_0}}|f|^2e^{-\varphi_m(t_0,z_1,\cdots z_m)}=1$$
and $<u^{\otimes m}(t_0),f>=|u(t_0)|^m$.

By Theorem \ref{thm:Demailly OT}, there eixsts $F\in \mathcal O(\Omega_m)$ such that $F|_{D^m_{t_0}}=f$
and
\begin{equation}\label{eqn:OT Thm hodge>0 domain case}
\int_{\Omega_m}|F(t,z_1,\cdots,z_m)|^2e^{-\varphi_m(t,z_1,\cdots z_m)}\leq C,
\end{equation}
where $C$ is a constant independent of $t_0$ and $m$.
Let $F_t(z_1,\cdots, z_m)=F(t,z_1,\cdots, z_m)$ and
$$||F_t||_t^2=\int_{D^m_t}|F_t|^2e^{-\varphi_m(t,z_1,\cdots,z_m)}.$$
Since $\varphi$ is bounded, by the mean value inequality, $\|F_t\|_t\leq +\infty$.
This implies $F_t$ lies in $E_t^{\hat\otimes m}$ for all $t\in U$ and hence $F$
can be seen as a holomorphic section of $E^{\hat\otimes m}$.

From the definition of $|u^{\otimes m}(t)|$, it is clear that
\begin{align*}
\|F_t\|_t|u(t)|^m\geq |\langle u^{\otimes m}(t), F_t\rangle|,
\end{align*}
and hence
$$e^{-m\log|u(t)|}\leq e^{-\log|<u^{\otimes m}(t),F_t>|}||F_t||_t.$$

Note that $<u^{\otimes m}(t),F_t>$ is a holomorphic function on $U$.
By Theorem \ref{thm:Demailly OT}, there is a holomorphic function $h$ on $U$ such that $h(t_0)=1$
and
\begin{equation}\label{eqn:OT on base Thm hodge>0 domain case}
\int_U |h(t)|^2e^{-2\log|<u^{\otimes m}(t),F_t>|}\leq C'e^{-2\log|<u^{\otimes m}(t_0),F_{t_0}>|}= C'e^{-2m\log|u(t_0)|},
\end{equation}
where $C'$ is a constant independent of $m$ and $t_0$.
So we have the estimate
\be
\begin{split}
     &\int_{U}|h(t)|e^{-m\log|u(t)|}\\
\leq &\int_{U}|h(t)|e^{-\log|<u^{\otimes m}(t),F_t>|}||F_t||_t\\
\leq &\left(\int_{U}|h(t)|^2|e^{-2\log|<u^{\otimes m}(t),F_t>|}\int_U||F_t||^2_t\right)^{1/2}\\
\leq &\sqrt{CC'}e^{-m\log|u(t_0)|},
\end{split}
\ee
where the last inequality follows from \eqref{eqn:OT Thm hodge>0 domain case}, \eqref{eqn:OT on base Thm hodge>0 domain case} and Fubini theorem.
By Theorem \ref{thm:cha. of p.s.h function}, $\log|u(t)|$ is subharmonic.
\end{proof}

\subsection{For families of compact K\"ahler manifolds}
In this subsection, we study the positivity of the direct image sheaf
of the twisted relative canonical bundle associated to a family of compact K\"ahler manifolds.

Let $X, Y$ be K\"ahler  manifolds of  dimension $r+n$ and $r$ respectively,
and let $p:X\rightarrow Y$  be  a proper holomorphic map.
For $y\in Y$ let $X_y=p^{-1}(y)$ , which is a compact submanifold of $X$ of dimension $n$ if $y$ is a regular value of $p$.
Let $L$ be a holomorphic line bundle over $X$, and $h$ be a singular  Hermitian metric on $L$,
whose curvature current is semi-positive.
Let $K_{X/Y}$ be the relative canonical bundle on $X$.

Let $\mathcal E_k=p_*(kK_{X/Y}\otimes L\otimes \mathcal I_k(h))$
and $\tilde{\mathcal E}_k=p_*(kK_{X/Y}\otimes L)$ be the direct image sheaves on $Y$.
We can choose a proper analytic subset $A\subset Y$ such that:
\begin{itemize}
\item[(1)] $p$ is submersive over $Y\backslash A$,
\item[(2)] both $\mathcal E_k$ and $\tilde{\mathcal E}_k$ are locally free on $X\backslash A$,
\item[(3)] for $y\in Y\backslash A$, $E_{k,y}$ and $\tilde E_{k,y}$ are naturally identified with
$H^0(X_y,kK_{X_y}\otimes L|_{X_y}\otimes \mathcal I_k(h)|_{X_y})$ and $H^0(X_y,kK_{X_y}\otimes L|_{X_y})$ respectively,
\end{itemize}
where $E_k$ and $\tilde E_k$ are the vector bundles on $Y\backslash A$
associated to $\mathcal E_k$ and $\tilde{\mathcal E_k}$ respectively.
For $u\in\tilde E_{k,y}$, as in \S \ref{subsec:Bergman manifold-case},
the $k$-norm of $u$ is defined to be
$$H_k(u):=\|u\|_k=\left(\int_{X_y}|u|^{2/k}h^{1/k}\right)^{k/2}\leq +\infty.$$
Then $H_k$ is a Finsler metric on $\tilde E_k$,
whose restriction on $E_k$ gives a Finsler metric on $E_k$,
which will be also denoted by $H_k$.
In the case that $k=1$, we denote $\mathcal E_1, \tilde{\mathcal E_1}, E_1, \tilde E_1, H_1$
by $\mathcal E, \tilde{\mathcal E}, E, \tilde E, H$ respectively.
The following theorem says that $H$ is a positively curved singular Hermitian metric
in the coherent sheaf $\mathcal E$ (see Definition \ref{def:finsler on sheaf} for definition).

\begin{thm}\label{thm:Positivity of hodge compact case}
With the above assumptions and notations,
$H$ is a positively curved singular metric on $\mathcal E$.
\end{thm}
\begin{proof}
The proof splits into three steps.

\subsubsection*{Step 1}
We prove that $H$ is a positively curved singular Finsler metric on $\tilde E\ra U:=Y\backslash A$.
The argument is similar to that in the proof of Theorem \ref{thm:Positivity of hodge domain case}.
Let $u$ be a local holomorphic section of $\tilde E^*$.
By definition, we need to show that $\log|u|$ is a plurisubharmonic function.
Without loss of generality, in this step we can assume that $U=\mathbb B^r$ is
the unit ball and $u$ is a holomorphic section of $\tilde E^*$ on $U$.

For $m\geq 1$, let $X_m=\{(y, z_1, \cdots z_m); y\in U, z_1, \cdots, z_m\in X_y\}$
be the $m$-th fiber-product power of $X$.
The is a natural proper holomorphic submersion from $X_m$ to $U$,
which is denoted by $p_m:X_m\ra U$.
Let $X^m_y=p_m^{-1}(y)$ be the fiber over $y\in U$.

For $1\leq i\leq m$, we have a projection $\pi_i:X_m\ra X$ which sends $(y, z_1,\cdots, z_m)$ to $(y,z_i)$.
Let $L_m=\pi^*_1L\otimes\cdots\otimes\pi^*_mL$ and let $h_m$ be the singular Hermitian metric on $L_m$ induced from
the metric $h$ on $L$. Then the curvature current of $h_m$ is nonnegative.

Note that $H^0(X^m_y, K_{X^m_y}\otimes L_m|_{X^m_y})=H^0(X_y, K_{X_y}\otimes L_{X_y})^{\otimes m}$ for $y\in U$.
Indeed, this follows from Proposition \ref{prop:product property-manifold} by putting a smooth hermitian metric on $L$.
In particular, the dimension of $H^0(X^m_y, K_{X^m_y}\otimes L_m)$ is independent of $y\in U$.
Since $U$ is assumed to be the unit ball, we can identify $K_{p_m^{-1}(U)/U}$ with $K_{p_m^{-1}(U)}$
So $p_{m,*}(K_{X_{m}/U}\otimes L_m)|_U$ is locally free and corresponds to a holomorphic vector bundle, say $\tilde E^m$, on $U$,
and we have $\tilde E^m=\tilde E^{\otimes m}$.
In the same way as defining $H$, we can define a Finsler metric,  say $H^m$, on $\tilde E^m$.
For $y\in U$, let $\tilde E_{y,b}$ and $\tilde E^{m}_{y,b}$ be subspaces of $\tilde E_y$ and $\tilde E^m_y$
consisting of vectors of finite norm.
By Lemma  \ref{lem:product property Bergman space}, we have $\tilde E^{m}_{y,b}=(\tilde E_{y,b})^{\otimes m}$.

Recall that $u$ is a holomorphic section of $\tilde E^*$ on $U$,
and we need to prove that $\log|u|$ is a plurisubharmonic function on $U$.
Note that the restriction $u|_E$ of $u$ on $E$ is a holomorphic section of $E^*$.
The point is that, by definition of the dual norm in Definition \ref{def:dual finsler metric},
the norm of $u|_E$ and $u$ are equal.
Therefore, by Proposition \ref{prop:dual Hodge metric upper semi-cont.:compact}, $\log|u|$ is upper semicontinuous.
Now it suffices to prove that $\log|u|$ satisfies the condition in Theorem \ref{thm:cha. of p.s.h function}.

Let $y_0\in U$ be any given point such that $|u(y_0)|\neq 0$.
$u^{\otimes m}$ is a holomorphic section of $(\tilde E^*)^{\otimes m}={\tilde E^{m*}}$.
Note that the definition of the norm of $u^{\otimes m}$ only involves vectors in $\tilde E^m$ of finite norm,
by Lemma  \ref{lem:product property Bergman space}, we have $|u^{\otimes m}(y)|=|u(y)|^m$.
There exists $f_{y_0}\in \tilde E^{m}_{y_0}$ such that
$\|f_{y_0}\|:=H^m(f_{y_0})=1$ and $|<u^{\otimes m}(y_0),f_{y_0}>|=|u(y_0)|^m$.
By Theorem \ref{thm:Demailly OT},
there is $F\in H^0(p_m^{-1}(U), (K_{X_m}\otimes L_m)|_{p_m^{-1}(U)})$ such that $F|_{X^m_{y_0}}=f_{y_0}$
and
$$\int_{X_m}|F(y,z_1,\cdots,z_m)|^2e^{-\varphi_m(z,z_1,\cdots, z_m)}\leq C,$$
where $\varphi_m$ is the weight of $h_m$ and $C$ is an absolute constant independent of $y_0$ and $m$.
For $y\in U$, let $F_y(z_1,\cdots, z_m)=F(y,z_1,\cdots, z_m)$, then $F_y\in \tilde E^m_{y}$ and
$$\|F_y\|^2=\int_{X^m_y}|F_y|^2e^{-\varphi_m(y,z_1,\cdots,z_m)}.$$
From the definition of $|u^{\otimes m}(y)|$, it is clear that
\begin{align*}
\|F_y\||u(y)|^m\geq |\langle u^{\otimes m}(y), F_y\rangle|,
\end{align*}and hence
$$e^{-m\log|u(y)|}\leq e^{-\log|<u^{\otimes m}(y),F_y>|}||F_y||.$$

Note that $F$ can be seen as a holomorphic section of $\tilde E^m$ on $U$,
so $<u^{\otimes m}(y),F_y>$ is a holomorphic function on $U$.
By Ohsawa-Takegoshi extension theorem (Theorem \ref{thm:Demailly OT}),
there is a holomorphic function $h$ on $U$ such that $h(t_0)=1$ and
$$\int_U |h(y)|^2e^{-2\log|<u^{\otimes m}(y),F_y>|}\leq C'e^{-2m\log|u(y_0)|},$$
where $C'$ is an absolute  constant independent of $m$ and $y_0$.
So we have the estimate
\be
\begin{split}
    &\int_{U}|h(y)|e^{-m\log|u(y)|}\\
    \leq &\int_{U}|h(y)|e^{-\log|<u^{\otimes m}(y),F_y>|}||F_y||\\
    \leq &\left(\int_{U}|h(y)|^2|e^{-2\log|<u^{\otimes m}(y),F_y>|}\int_U||F_y||^2\right)^{1/2}\\
    \leq &\sqrt{CC'}e^{-m\log|u(y_0)|}.
\end{split}
\ee
By Theorem \ref{thm:cha. of p.s.h function}, $\log|u(t)|$ is plurisubharmonic.

\subsubsection*{Step 2}
We prove that $H$ is a positively curved singular Finsler metric on $E\ra U:=Y\backslash A$.
We also assume that $U=\mathbb B^r$ be the unit ball.
Note that $E$ is a holomorphic subbundle of $\tilde E$.
Since $U$ is a Stein manifold, there is a holomorphic subbundle $E'$ of $\tilde E$
such that $\tilde E$ splits as $E\oplus E'$ (see e.g. Corollary 2.4.5 in \cite{Forstneric11}).
So any holomorphic section $u$ of $E^*$ on $U$ can be extended to a holomorphic section $\tilde u$
of $\tilde E^*$ by setting $u(a)=0$ for all $a\in E'$.
Note that the norm of any vector in $\tilde E\backslash E$ is $+\infty$ (by Theorem \ref{thm:Demailly OT}),
by definition, the norm of $u$ and $\tilde u$ are equal.
By the result in Step 1, $\log|\tilde u|$ is plurisubharmonic,
so $\log|u|$ is plurisubharmonic.

\subsubsection*{Step 3}
We will complete the proof of Theorem \ref{thm:Positivity of hodge compact case} in this final step.
Let $u$ be a holomorphic section of the dual sheaf $\mathcal E^*$ of $\mathcal E$ on some open set $V$ in $Y$.
We want to show that $|u|_{V\backslash A}$ is bounded above on all compact subsets of $V$.
Once this is established, $\log|u|$ can be extended uniquely to a plurisubharmonic function on $V$ and we are done.

The proof of the boundeddness of $\log|u|$ follows from the idea in the proof of Proposition 23.3 in \cite{HPS16}.
Without loss of generality, we assume that $V=\mathbb B^r$ is the unit ball.
For any $y\in V\backslash A$ such that $|u(y)|\neq 0$ (otherwise there is nothing to prove),
there is $a\in E_y$ such that $|a|=1$ and $<u(y),a>=|u(y)|$.
By Theorem \ref{thm:Demailly OT}, there exists a holomorphic section $s$ of $\mathcal E$ on $V$
such that $s(y)=a$ and
$$\int_{V}|s|^2_h\leq C,$$
where $C$ is a constant independent of $y$ and $a$.
Let
$$S=\{f\in H^0(V,\mathcal E); \int_{V}|f|^2_h\leq C\}.$$
Since the metric $h$ on $L$ is lower semicontinuous,
by the mean value inequality and Montel theorem, $S$ is a normal family,
namely, any sequence in $S$ has a subsequence that converges uniformly on compact subsets of $V$.
Note that if $s_j$ is a subsequence of $S$ that converges uniformly on compact subsets of $V$,
then the sequence of holomorphic functions $<u,s_j>$ converges on compact subsets of $V$,
and hence is uniformly bounded on compact sets of $V$.
So $\{<u,s>;s\in S\}$ is uniformly bounded on compact sets of $V$.
\end{proof}

\begin{rem}
By Theorem \ref{thm:p.s.h variation Berg compact} (and the remark following it) and Theorem \ref{thm:Positivity of hodge compact case},
one can see that the NS metric (see \S \ref{subsec:Bergman relative case} for definition) on the direct image $p_*(kK_{X/Y}\otimes L\otimes \mathcal I_k(h))$
is positively curved in the sense of Griffiths.
\end{rem}

\end{document}